\documentclass[draft,twoside,10pt,leqno]{amsart}
\usepackage{srcltx}
\usepackage[english]{babel}
\numberwithin{equation}{section}
\newtheorem{thm}{Theorem}[section]

\newtheorem*{thmA}{Theorem A}
\newtheorem*{thmB}{Theorem B}%[section]

\newtheorem*{thmA'}{Theorem A'}
\newtheorem*{thmB'}{Theorem B'}%[section]

\newtheorem*{thmA''}{Theorem A''}
\newtheorem*{thmB''}{Theorem B''}%[section]
\newtheorem{lem}[thm]{Lemma}
\newtheorem{pro}[thm]{Proposition}
\newtheorem{cor}[thm]{Corollary}

\theoremstyle{definition}
\newtheorem{defi}[thm]{Definition}

\newtheorem{rmk}[thm]{Remark}

\numberwithin{equation}{section}

\usepackage{amssymb,amsmath}
\usepackage{mathtools}

\def\div{\mathop\mathrm{div}\nolimits}

\def\tr{\mathop\mathrm{tr}\nolimits}

\def\vol{\mathop\mathrm{vol}\nolimits}

\def\Ric{\mathop\mathrm{Ric}\nolimits}

\newcommand{\supp}{\operatorname{supp}}
\newcommand{\dist}{\operatorname{dist}}

\newcommand{\e}{\mbox{$\mathrm{e}$}}
\newcommand{\R}{\mbox{${\mathbb R}$}}
\newcommand{\N}{\mbox{${\mathbb N}$}}

\newcommand{\g}[2]{\langle #1 ,#2 \rangle}

\newcommand{\ep}{\varepsilon}
\newcommand{\rf}[1]{\mbox{(\ref{#1})}}

\newcommand{\eig}{\lambda_1^L}
\newcommand{\eigd}{\lambda_1^{\Delta}}
\newcommand{\eigb}{\lambda_1^{\overline{L}}}
\newcommand{\M}{(M,\g{\ }{\ })}
\newcommand{\Mo}{M\setminus\left\{o\right\}}
\newcommand{\MG}{\left(M,\g{\,}{\,},G\right)}
\newcommand{\smM}{\mathrm{C}^{\infty}(M)}
\def\beq{\begin{equation}}
\def\eeq{\end{equation}}
\def\beqs{\begin{equation*}}
\def\eeqs{\end{equation*}}

\begin{document}

\title[Lichnerowicz-type equations on complete manifolds]{Lichnerowicz-type equations\\ on complete manifolds}

\author[Guglielmo Albanese]{Guglielmo Albanese}
\author[Marco Rigoli]{Marco Rigoli}

\address{Dipartimento di Matematica, Universit\`{a} degli Studi di Milano, Via Saldini 50, I-20133, Milano, Italy}
\email {guglielmo.albanese@unimi.it, gu.albanese@gmail.com}
\email {marco.rigoli@unimi.it}

\address{Departamento de Matem\'atica, Universidade Federal do Cear\'a
Av. Humberto Monte s/n, Bloco 914, 60455-760 Fortaleza, Brazil}
\email {marco.rigoli55@gmail.it}

\date{\today}

\maketitle
\begin{abstract}
Under appropriate spectral assumptions we prove two existence results for positive solutions of Lichnerowicz-type equations on complete manifolds. We also give \emph{a priori} bounds and a comparison result that immediately yields uniqueness for certain classes of solutions. No curvature assumptions are involved in our analysis.
\end{abstract}
\section{Introduction}

In the analysis of Einstein field equations in General Relativity the initial data set for the non-linear wave system plays an essential role. These initial data have to satisfy the Einstein constraint conditions that can be expressed in a geometric form as follows. Let $\left(M,\widehat{g}\right)$ be a Riemannian manifold and $\widehat{K}$ a symmetric $2$-covariant tensor on $M$. Then $\left(M,\widehat{g}\right)$ is said to satisfy the Einstein constraint equations with non-gravitational energy density $\widehat{\rho}$ and non-gravitational momentum density $\widehat{J}$ if
	\beq\label{inidata}
	\begin{cases}
	\left|\widehat{K}\right|^2_{\widehat{g}}-\left(\tr_{\widehat{g}}\widehat{K}\right)^2=S_{\widehat{g}}-\widehat{\rho}&\\
	\div_{\widehat{g}}\widehat{K}-\nabla\tr_{\widehat{g}}\widehat{K}=\widehat{J}\,.&
	\end{cases}
	\eeq
Here $S_{\widehat{g}}$ stands for the scalar curvature of the metric $\widehat{g}$. To look for solutions of \rf{inidata} using the conformal method introduced by Lichnerowicz in \cite{Li} means that we generate an initial data set $\left(M,\widehat{g},\widehat{K},\widehat{\rho},\widehat{J}\right)$ satisfying \rf{inidata} by first choosing the following conformal data: A Riemannian manifold $\M$; a symmetric $2$-covariant tensor $\sigma$ required to be traceless and transverse with respect to $\g{\,}{\,}$, that is, for which $\tr_{\g{\,}{\,}}\sigma=0$ and $\div_{\g{\,}{\,}}\sigma=0$; a scalar function $\tau$ (that will play the role of a non normalized mean curvature); a non-negative scalar function $\rho$ and a vector field $J$. Then one looks for a positive function $u$ and a vector field $W$ that solve the conformal constraint equations
	\beq\label{confcon}
	\begin{cases}
	\Delta u-c_{m}Su+c_m\left|\sigma+\mathring{\mathcal{L}}W\right|^{2}u^{-N-1}-b_m\tau^2u^{N-1}+c_m\rho u^{-N\slash2}=0&\\
	\Delta_{\mathbb{L}}W+\frac{m-1}{m}u^N\nabla\tau+J=0\,.&
	\end{cases}
	\eeq
Where $\Delta$, $S$, and $\left|\cdot\right|$ denote respectively the Laplace-Beltrami operator, the scalar curvature, and the norm in the metric $\g{\,}{\,}$. The operator $\mathring{\mathcal{L}}$ is the traceless Lie derivative, that is, in a local orthonormal coframe
	\beqs
	\left(\mathring{\mathcal{L}}W\right)_{ij}=W_{ij}-W_{ji}-\frac{2}{m}\left(\div W\right)\delta_{ij}
	\eeqs 
and $\Delta_{\mathbb{L}}=\div\circ\mathring{\mathcal{L}}$ is the vector laplacian. The constants appearing in \rf{confcon} are respectively given by
	\beqs
	N=\frac{2m}{m-2}\,,\quad\quad c_m=\frac{m-2}{4(m-1)}\,,\quad\quad b_m=\frac{m-2}{4m}\,,
	\eeqs
in particular we note that $N$ is the critical Sobolev exponent. If $\left(u,W\right)$ is a solution of \rf{confcon} then the initial data set
	\beqs
	\begin{aligned}
	\widehat{g}=u^{\frac{4}{m-2}}\g{\,}{\,}\,,\quad\quad \widehat{\rho}=u^{-\frac{3N+2}{2}}\rho\,,\quad\quad\widehat{J}=u^{-N}J\,,\\
	\widehat{K}^{ij}=u^{-2\frac{m+2}{m-2}}\left(\sigma+\mathring{\mathcal{L}}W\right)^{ij}+\frac{\tau}{m}u^{\frac{4}{m-2}}\delta^{ij}
	\end{aligned}
	\eeqs
satisfies the Einstein constraints \rf{inidata}. For further informations on the physical content of the system \rf{confcon} we refer to the recent surveys \cite{BI}, \cite{CGP}, and the references therein.\\

The aim of this paper is to study the existence, \emph{a priori} bounds, and uniqueness of positive solutions of the Lichnerowicz-type equation
	\beq\label{lich}
	\Delta u+a(x)u-b(x)u^{\sigma}+c(x)u^{\tau}=0
	\eeq
with, at least, continous coefficients, $a(x)$, $b(x)$, $c(x)$, $\sigma>1$, $\tau<1$, and with the sign restrictions
	\beq\label{signbc}
	b(x)\geq 0\,,\quad\quad c(x)\geq0\,,
	\eeq
on a complete, non compact, connected manifold $\M$.
Equation \rf{lich} is of the same form of the scalar equation of system \rf{confcon} in case $\rho\equiv0$, the coefficients $b(x)$ and $c(x)$ corresponding respectively to
	\beqs
	b_m\tau^2\quad\hbox{and}\quad c_m\left|\sigma+\mathcal{L}W\right|^2_g\,.
	\eeqs
This latter fact justifies the sign condition \rf{signbc}. %It is also worth observing that, beside nonnegativity of $c(x)$, we will not add other requests on it for the validity of the existence result. This fact should simplify the search for solutions of the coupled system \rf{confcon}.\\
In recent years this type of equations has been studied by many authors, see for instance \cite{Ma}, \cite{Br}, \cite{MW}, and \cite{DAM}. In the present work we significantly generalize many of the results obtained in the aforementioned papers.\\
We now introduce some notations and state our main existence results. Let $a(x)\in\mathrm{C}^0(M)$ and $L=\Delta+a(x)$. If $\Omega$ is a non-empty open set, the first Dirichlet eigenvalue $\lambda_1^L(\Omega)$ is variationally characterized by means of the formula
	\beq\label{eivar}
	\lambda_1^L(\Omega)=\inf\left\{\int_{\Omega}\left|\nabla\varphi\right|^2-a(x)\varphi^2\,:\,\varphi\in\mathrm{W}_0^{1,2}(\Omega)\,,\int_{\Omega}\varphi^2=1\right\}\,.
	\eeq
We recall that, if $\Omega$ is bounded and $\partial\Omega$ is sufficiently regular, the infimum is attained by the unique normalized eigenfunction $v$ on $\Omega$ satisfying
	\beq\label{eifun}
	\begin{cases}
	\Delta v+a(x)v+\lambda_1^L(\Omega)v=0 & \hbox{on $\Omega$}\\
	v=0 & \hbox{on $\partial\Omega$}\\
	v>0 & \hbox{on $\Omega$}.
	\end{cases}
	\eeq
We then define the first eigenvalue of $L$ on $M$ as
	\beqs
	\eig(M)=\inf_{\Omega}\eig(\Omega)
	\eeqs
where $\Omega$ runs over all bounded domains of $M$. Observe that, due to the monotonicity of $\eig$ with respect to the domain, that is
	\beq\label{eimon}
	\Omega_1\subseteq\Omega_2\quad\hbox{implies}\quad\eig(\Omega_1)\geq\eig(\Omega_2)\,,
	\eeq
we have
	\beq
	\eig(M)=\lim_{r\rightarrow+\infty}\eig(B_r)
	\eeq
where, from now on, $B_r$ denotes the geodesic ball in the complete manifold $\M$ of radius $r$ centered at a fixed origin $o\in M$.\\
Note that in case $\Omega_2\setminus\Omega_1$ has non-empty interior the inequality in \rf{eimon} becomes strict.\\
We need to extend definition \rf{eivar} to an arbitrary bounded subset $B$ of $M$. We do this by setting
	\beq\label{eiB}
	\eig(B)=\sup_{\Omega}\eig(\Omega)
	\eeq
where the supremum is taken over all open bounded sets $\Omega$ with smooth boundary such that $B\subseteq\Omega$ . Observe that, by definition, if $B=\emptyset$ then $\eig(B)=+\infty$.\\

We would like to remark that since the first Dirichlet eigenvalue for the Laplacian of a ball $B_r$ grows like $r^{-2}$ as $r\rightarrow 0^+$, $\eig(B_r)\geq0$ provided $r$ is sufficiently small and one may think of $\eig(B)>0$ as a condition expressing the fact that $B$ is small in a spectral sense. Of course, this condition also depends on the behaviour of $a(x)$ therefore \emph{small in a spectral sense} does not necessarily mean, for instance, \emph{small in a Lebesgue measure sense}. This is clear if $a(x)\leq0$ because $\lambda_1^{\Delta}(M)\geq0$ in any complete manifold $\M$ so that in this case $\eig(M)\geq0$ and thus $\eig(B)>0$ on any bounded set $B\subset M$.\\

The main results of the paper are the following two existence theorems for positive solutions of equation \rf{lich}.

	\begin{thmA}
	Let $a(x)$, $b(x)$, $c(x)\in\mathrm{C}_{loc}^{0,\alpha}(M)$ for some $0<\alpha\leq1$. Assume \rf{signbc} and suppose that $b(x)$ is strictly positive outside some compact set. Let
		\beq\label{B0}
		B_0=\left\{x\in M:\, b(x)=0\right\}
		\eeq
	and suppose
		\beq\label{Bo>0}
		\eig(B_0)>0
		\eeq
	with $L=\Delta+a(x)$. Assume that 
		\beq\label{hpthmA}
		\eig(M)<0\,,
		\eeq
	then \rf{lich} has a maximal positive solution $u\in\mathrm{C}^2(M)$. 
	\end{thmA}
In this case \emph{maximal} means that if $0<w\in\mathrm{C}^2(M)$ is a second solution of \rf{lich}, then $w\leq u$.
In the same vein we have the following result, where the spectral condition \rf{hpthmA} is substituted by a spectral smallness requirement on the zero set of the coefficient $c(x)$ and a pointwise control on the coefficients. Recall that given the real function $\alpha(x)$, its positive and negative part are respectively defined by
	\beqs
	\begin{aligned}
	& \alpha_+(x)=\max\left\{\alpha(x),0\right\}\,,\\
	& \alpha_-(x)=-\min\left\{\alpha(x),0\right\}\,.	
	\end{aligned}
	\eeqs

	\begin{thmB}
	Let $a(x)$, $b(x)$, $c(x)\in\mathrm{C}_{loc}^{0,\alpha}(M)$ for some $0<\alpha\leq1$. Assume the validity of \rf{signbc} and let
		\beqs
		B_0=\left\{x\in M:\, b(x)=0\right\}\,,\quad\quad C_0=\left\{x\in M:\, c(x)=0\right\}\,,
		\eeqs
	and
		\beq\label{hp1thmB}
		\lambda_1^{\Delta+a}(B_0)>0\,.
		\eeq
	Suppose that there exist two bounded open sets $\Omega_1$, $\Omega_2$ such that $C_0\subset\Omega_1\subset\subset\Omega_2$,
		\beq\label{hp2thmB}
		\sup_{M\setminus\overline{\Omega}_1}\frac{a_-(x)+b(x)}{c(x)}<+\infty \,,		
		\eeq
	and
		\beq\label{hp3thmB}
		\lambda_1^{\Delta-a}(\Omega_2)>0\,.
		\eeq
	Then \rf{lich} has a maximal positive solution $u\in\mathrm{C}^2(M)$. 
	\end{thmB}
	
The proof of the theorems is the content of Section 2. These existence results should be compared with those obtained at the end of the very recent paper \cite{MW}.\\
The remaining two sections of the paper are devoted to uniqueness of solutions. In particular, the results of Section 3 should be interpreted as \emph{Liouville-type theorems} and compared with those obtained in \cite{Ma}, \cite{Br}, and \cite{DAM}. The main differences with previous work in the literature is that our geometric requirement on the manifold consist only in a mild volume growth assumption for geodesic balls and in the fact that we allow for non constant coefficients $a(x)$, $b(x)$, $c(x)$ in equation \rf{lich}. In this last setting in general there are no \emph{trivial} solutions at hand. Thus, to provide a complete analysis of the problem in this case, we need to find an \emph{a priori} estimate and use it to detect a \emph{trivial} solution of \rf{lich}. In particular, Corollary \ref{corLio} is our main \emph{Liouville-type theorem}.\\
In Section 4 we analyze another uniqueness result, this time under a spectral assumption on the manifold, in the spirit of the very recent \cite{BMR} and \cite{BMR1}.

\section{Proof of Theorems A and B}
The main result of the paper is in fact the following proposition, whose proof will be the content of the section. At the end we will prove Theorems A and B as corollaries of the proposition.
	\begin{pro}\label{thmmax}
	Let $a(x)$, $b(x)$, $c(x)\in\mathrm{C}_{loc}^{0,\alpha}(M)$ for some $0<\alpha\leq1$. Assume \rf{signbc} and suppose that $b(x)$ is strictly positive outside some compact set. Furthermore, suppose
		\beq\label{Bo>0pr}
		\eig(B_0)>0
		\eeq
	with $L=\Delta+a(x)$. If $0<u_-\in\mathrm{C}^0(M)\cap\mathrm{W}^{1,2}_{loc}(M)$ is a global subsolution of \rf{lich} on $M$, then \rf{lich} has a maximal positive solution $u\in\mathrm{C}^2(M)$.
	\end{pro}
The proof of Proposition \ref{thmmax} is divided into several steps. In what follows we keep the notations of the proposition.
	\begin{lem}\label{lem1}
	Let $\overline{a}(x)$, $b(x)$, $c(x)\in\mathrm{C}_{loc}^{0,\alpha}(M)$ for some $0<\alpha\leq1$, and let \rf{signbc} hold. Suppose that $B_0$ is bounded and
		\beq
		\lambda_1^{\overline{L}}(B_0)>0
		\eeq
	where $\overline{L}=\Delta+\overline{a}(x)$. If $\Omega$ is a bounded open set such that $B_0\subset\Omega$, then there exists $v_+$ solution of
		\beq\label{v+}
		\begin{cases}
		\Delta v_++\overline{a}(x)v_+-b(x)v_+^{\sigma}+c(x)v_+^{\tau}\leq0 & \hbox{on $\Omega$}\\
		v_+>0 & \hbox{on $\overline{\Omega}$}.
		\end{cases}
		\eeq
	\end{lem}
		\begin{proof}
		Let $D$ and $D'$ be bounded open domains such that
			\beqs
			B_0\subset\subset D'\subset\subset D\subset\subset\Omega\,,
			\eeqs
		and $\lambda_1^{\overline{L}}(D)>0$. Let $u_1$ be a positive solution of 
			\beqs
			\begin{cases}
			\Delta u_1+\overline{a}(x)u_1+\lambda_1^{\overline{L}}(D)u_1=0 & \hbox{on $D$}\\
			u_1=0 & \hbox{on $\partial D$}.
			\end{cases}
			\eeqs
		Since $b(x)>0$ on $M\setminus B_0$ and $\overline{\Omega}\setminus D'\subset\subset M\setminus B_0$,
			\beqs
			\beta=\inf_{\overline{\Omega}\setminus D'}b(x)>0\,.
			\eeqs
		Define
			\beqs
			\alpha=\sup_{\overline{\Omega}\setminus D'}\overline{a}(x)\,, \quad\quad \delta=\sup_{\overline{\Omega}\setminus D'}c(x)\,,
			\eeqs
		and note that $\alpha$, $\delta<+\infty$ since $\Omega$ is bounded. Let $U$ be a positive constant. Then
			\beqs
			\begin{aligned}
			\Delta U+\overline{a}(x)U-b(x)U^{\sigma}+c(x)U^{\tau} & = U\left(\overline{a}(x)-b(x)U^{\sigma-1}+c(x)U^{\tau-1}\right)\\
			& \leq U\left(\alpha-\beta U^{\sigma-1}+\delta U^{\tau-1}\right)
			\end{aligned}
			\eeqs
		on $\Omega\setminus D'$. We observe that the RHS of the above is non-positive for $U$ sufficiently large, say
			\beq\label{Lambda0}
			U\geq\Lambda_0>0\,.
			\eeq
		Next we choose a cut-off function $\psi\in\mathrm{C}_0^{\infty}(M)$ such that $0\leq\psi\leq1$, $\psi\equiv1$ on $D'$, and $\supp\psi\subset D$.
		Fix a positive constant $\gamma$ and define
			\beq
			u=\gamma\left(\psi u_1+(1-\psi)\Lambda_0\right)\,.
			\eeq
		Since $b(x)\geq0$ and $\eigb(D)>0$, on $\overline{D'}$ we have
			\beqs
			\begin{aligned}
			\overline{L}u-b(x)u^{\sigma}+c(x)u^{\tau} & = \overline{L}\left(\gamma u_1\right)-b(x)\left(\gamma u_1\right)^{\sigma}+c(x)\left(\gamma u_1\right)^{\tau}\\
			& = -\left[\eigb(D)\left(\gamma u_1\right)+b(x)\left(\gamma u_1\right)^{\sigma}-c(x)\left(\gamma u_1\right)^{\tau}\right]\\
			& = -\left(\gamma u_1\right)\left[\eigb(D)+b(x)\left(\gamma u_1\right)^{\sigma-1}-c(x)\left(\gamma u_1\right)^{\tau-1}\right].
			\end{aligned}
			\eeqs
		For the RHS of the above to be non-positive on $\overline{D'}$ it is sufficient to have
			\beq\label{gamlam}
			c(x)\left(\gamma u_1\right)^{\tau-1}\leq\eigb(D)\quad\quad\hbox{on $\overline{D'}$}.
			\eeq
		Towards this aim we note that, since $\overline{D'}$ is compact, $u_1>0$ on $\overline{D'}$, and $\tau<1$, then \rf{gamlam} is satisfied for
			\beq\label{Gamma0}
			\gamma\geq\Gamma_0=\Gamma_0(u_1)>0\,,
			\eeq
		sufficiently large.	We now consider $\Omega\setminus D$, since $\supp\psi\subset D$, it follows that $u=\gamma\Lambda_0$ there. Thus, using $\Omega\setminus D\subset\Omega\setminus D'$, from \rf{Lambda0} it follows that
			\beqs
			\Delta u+\overline{a}(x)u-b(x)u^{\sigma}+c(x)u^{\tau}\leq0\quad\quad\hbox{on $\Omega\setminus D$}\,,
			\eeqs 
		is satisfied if we choose $\gamma\geq 1$, indeed in this case
			\beqs
			\gamma\Lambda_0\geq\Lambda_0\,.
			\eeqs
		It remains to analyze the situation on $D\setminus\overline{D'}$. First of all we note that, by standard elliptic regularity theory, $u_1\in\mathrm{C}^2(D)$. Thus, since $\supp\psi\subset D$, it follows that $u\in\mathrm{C}^2(\Omega)$, in particular this implies that there exists a positive constant $C_0$ such that
			\beqs
			\left(\Delta+\overline{a}(x)\right)u\leq\gamma C_0\quad\quad\hbox{on $D\setminus\overline{D'}$}\,.
			\eeqs 
		Thus on $D\setminus\overline{D'}$ we have
			\beqs
			\begin{aligned}
			\Delta u+\overline{a}(x)u-b(x)u^{\sigma}+c(x)u^{\tau} & \leq \gamma C_0-b(x)\gamma^{\sigma}\left(\psi u_1+\left(1-\psi\right)\Lambda_0\right)^{\sigma}\\
			& \quad\quad+c(x)\gamma^{\tau}\left(\psi u_1+\left(1-\psi\right)\Lambda_0\right)^{\tau}\,.
			\end{aligned}
			\eeqs
		Now there exists constants $\ep$ and $E$ such that
			\beqs
			\begin{aligned}
			& \inf_{D\setminus\overline{D'}}b(x)\left(\psi u_1+\left(1-\psi\right)\Lambda_0\right)^{\sigma}=\ep>0\,,\\
			& \sup_{D\setminus\overline{D'}}c(x)\left(\psi u_1+\left(1-\psi\right)\Lambda_0\right)^{\tau}=E<+\infty\,.
			\end{aligned}
			\eeqs
		Therefore, on $D\setminus\overline{D'}$
			\beqs
			\Delta u+\overline{a}(x)u-b(x)u^{\sigma}+c(x)u^{\tau}\leq\gamma\left(C_0-\ep\gamma^{\sigma-1}+E\gamma^{\tau-1}\right)\,.
			\eeqs
		Since $\sigma>1$ and $\tau<1$, it follows that there exists a positive constant $\Gamma_1$ depending only on $D$ and $D'$ such that
			\beqs
			C_0-\ep\gamma^{\sigma-1}+E\gamma^{\tau-1}\leq0
			\eeqs
		for $\gamma\geq\Gamma_1$.\\
		Thus, by choosing 
			\beqs
			\gamma\geq\max\left\{1,\Gamma_0,\Gamma_1\right\}
			\eeqs
		$u$ is the desidered supersolution $v_+$ of \rf{v+} on $\Omega$.
		\end{proof}
		
	\begin{defi}\label{defsigma}
	We say that the \emph{property $\left(\Sigma\right)$} holds on $M$ (for equation \rf{lich}) if there exists $R_o\in\R^+$ such that for all $R\geq R_o$ there exists a solution $\varphi\in\mathrm{C}^0(\overline{B_R})\cap\mathrm{W}^{1,2}_{loc}(B_R)$ of
		\beq\label{sigpro}
		\begin{cases}
		\Delta \varphi+a(x)\varphi- b(x)\varphi^{\sigma}+c(x)\varphi^{\tau}\geq 0 & \hbox{on $B_{R}$}\\
		\varphi\geq 0 & \hbox{on $\overline{B_{R}}$}\,.
		\end{cases}
		\eeq
	When $\tau<0$, in order to avoid singularities, in the equation above it is assumed that $\supp c(x)\subseteq\supp \varphi$.\\
	In Proposition \ref{suffsigma} below we shall give some sufficient conditions for the validity of property $\left(\Sigma\right)$.
	\end{defi}
	
	\begin{lem}\label{lem2}
	Let $a(x)$, $b(x)$, $c(x)\in\mathrm{C}_{loc}^{0,\alpha}(M)$ for some $0<\alpha\leq1$, and let \rf{signbc} hold. Suppose that $B_0$ is bounded and $\eig(B_0)>0$ with $L=\Delta+a(x)$. Furthermore assume that property $\left(\Sigma\right)$ holds on $M$. Let $\Omega$ be a bounded domain such that $B_0\subset\Omega$. Then, for each $n\in\N$, there exists a solution $u$ of the problem
		\beq\label{nlem1}
		\begin{cases}
		\Delta u+a(x)u-b(x)u^{\sigma}+c(x)u^{\tau}=0 & \hbox{on $\Omega$}\\
		u>0 & \hbox{on $\Omega$}\\
		u=n & \hbox{on $\partial\Omega$}.
		\end{cases}
		\eeq
	\end{lem}
		\begin{proof}
		By the definition of $\eig(B_0)$ and the assumption of positivity, there exists an open domain $D$ with smooth boundary such that $B_0\subset D\subset\subset\Omega$ and $\eig(D)>0$. Let $\rho\in\smM$ be a cut-off function such that $0\leq\rho\leq 1$, $\rho\equiv1$ on $D$, $\rho\equiv0$ on $M\setminus\overline{\Omega}$. Fix $N\geq\max\left\{\sup_{\overline{\Omega}}\left|a(x)\right|+1\,,\eigd(M\setminus\overline{\Omega})+1\right\}$.\\
		Define
			\beqs
			\overline{a}(x)=\rho(x)a(x)+N\left(1-\rho(x)\right)
			\eeqs
		and consider the operator $\overline{L}=\Delta+\overline{a}(x)$. Since $\overline{a}(x)=a(x)$ on $D$,
			\beq\label{LB0}
			\eigb(B_0)=\eig(B_0)>0\,.
			\eeq
		Furthermore, from $N\geq\eigd(M\setminus\overline{\Omega})+1$, we deduce $\eigb(M\setminus\overline{\Omega})\leq-1$ and it follows that there exists $R>0$ sufficiently large such that 	
			\beq
			\overline{\Omega}\subset B_R\quad\quad\hbox{and}\quad\quad\eigb(B_R)<0\,.
			\eeq
		Fix $\ep>0$. Then $\eigb(B_{R+\ep})<0$. Let $\varphi$ be the normalized eigenfunction of $\overline{L}$ on $B_{R+\ep}$ relative to the eigenvalue $\eigb(B_{R+\ep})$ (here, without loss of generality, that is, possibly substituting $B_{R+\ep}$ with a slighly larger open set with smooth boundary, we are supposing $\partial B_{R+\ep}$ smooth) so that
			\beq
			\begin{cases}
			\overline{L}\varphi+\eigb(B_{R+\ep})\varphi=0 & \hbox{on $B_{R+\ep}$}\\
			\varphi\equiv 0 & \hbox{on $\partial B_{R+\ep}$}\\
			\varphi>0 & \hbox{on $B_{R+\ep}$}\\
			\|\varphi\|_{\mathrm{L}^2(B_{R+\ep})}=1 & \\
			\end{cases}
			\eeq
		We fix $\gamma>0$ sufficiently small that
			\beqs
			\int_{B_{R+\ep}}\left|\nabla\varphi\right|^2-\overline{a}(x)+\gamma\left[b(x)-c(x)\right]\varphi^2=\eigb(B_{R+\ep})+\int_{B_{R+\ep}}\gamma\left[b(x)-c(x)\right]\varphi^2<0\,.
			\eeqs
		This shows that the operator $\widetilde{L}=\Delta+\overline{a}(x)-\gamma\left[b(x)-c(x)\right]$ satisfies $\lambda_1^{\widetilde{L}}(B_{R+\ep})<0$. Let $\psi$ be a positive eigenfunction corresponding to $\lambda_1^{\widetilde{L}}(B_{R+\ep})$. The $\psi$ satisfies
			\beqs
			\begin{cases}
			\Delta\psi+\overline{a}(x)\psi-\gamma b(x)\psi+\gamma c(x)\psi\geq 0 & \hbox{on $B_{R+\ep}$}\\
			\psi\equiv 0 & \hbox{on $\partial B_{R+\ep}$}\\
			\psi>0 & \hbox{on $B_{R+\ep}$}\,.
			\end{cases}
			\eeqs
		Thus
			\beq\label{psiR}
			\begin{cases}
			\Delta\psi+\overline{a}(x)\psi-\gamma b(x)\psi+\gamma c(x)\psi\geq 0 & \hbox{on $B_{R}$}\\
			\psi>0 & \hbox{on $\overline{B_{R}}$}\,.
			\end{cases}
			\eeq
		Let $\mu>0$ and define $v_-=\mu\psi$ on $B_R$. Choosing 
			\beqs
			\mu\leq\min\left\{\gamma^{\frac{1}{\sigma-1}}\left(\sup_{B_R}\psi\right)^{-1},\,\gamma^{\frac{1}{\tau-1}}\left(\sup_{B_R}\psi\right)^{-1}\right\}
			\eeqs
		we have	
			\beq\label{psiiii}
			\hbox{i) } \gamma\mu^{1-\sigma}\psi^{1-\sigma}\geq 1\quad\quad\hbox{and}\quad\quad\hbox{ii) } \gamma\mu^{1-\tau}\psi^{1-\tau}\leq 1
			\eeq
		on $B_R$. Hence, using \rf{psiR} and \rf{psiiii} we deduce
			\beqs
			\begin{aligned}
			0 & \leq\Delta v_-+\overline{a}(x)v_--\gamma b(x)v_-^{\sigma}\left(\mu^{1-\sigma}\psi^{1-\sigma}\right)+\gamma c(x)v_-^{\tau}\left(\mu^{1-\tau}\psi^{1-\tau}\right)\\
			& \leq 	\Delta v_-+\overline{a}(x)v_-- b(x)v_-^{\sigma}+c(x)v_-^{\tau}\,,
			\end{aligned}
			\eeqs
		that is,
			\beq\label{sublem2}
			\begin{cases}
			\Delta v_-+\overline{a}(x)v_-- b(x)v_-^{\sigma}+c(x)v_-^{\tau}\geq 0 & \hbox{on $B_{R}$}\\
			v_->0 & \hbox{on $\overline{B_{R}}$}\,.
			\end{cases}
			\eeq
		Because of the validity of \rf{LB0}, Lemma \ref{lem1} yields the existence of $v_+$ satisfying
			\beq\label{suplem2}
			\begin{cases}
			\Delta v_++\overline{a}(x)v_+- b(x)v_+^{\sigma}+c(x)v_+^{\tau}\leq 0 & \hbox{on $B_{R}$}\\
			v_+>0 & \hbox{on $\overline{B_{R}}$}\,.
			\end{cases}
			\eeq
		Note that if $0<\gamma\leq1$, $\gamma v_-$ still satisfies \rf{sublem2}; hence up to choosing a suitable $\gamma$ we can suppose that
			\beq
			\sup_{\overline{B_R}}v_-\leq \inf_{\overline{B_R}}v_+\quad\quad\hbox{on $\overline{B_R}$}\,.
			\eeq
		Let 
			\beqs
			\alpha_+=\inf_{\partial{B_R}}v_+\,,\quad\quad\alpha_-=\sup_{\partial{B_R}}v_-\,,
			\eeqs
		and fix $\alpha\in\left[\alpha_-,\alpha_+\right]$. Then, by the monotone interation scheme, there exists a solution $w$ of
			\beqs
			\begin{cases}
			\Delta w+\overline{a}(x)w- b(x)w^{\sigma}+c(x)w^{\tau}= 0 & \hbox{on $B_{R}$}\\
			w\equiv\alpha>0 & \hbox{on $\partial{B_{R}}$}\,,
			\end{cases}
			\eeqs
		and with the further property that 
			\beqs
			0< v_-\leq w\leq v_+\quad\quad\hbox{on $\overline{B_R}$}\,.
			\eeqs
		Therefore, since $\overline{a}(x)\geq a(x)$ on $\overline{B_R}$ and $w>0$ we have 
			\beq\label{wsuplem2}
			\begin{cases}
			\Delta w+a(x)w- b(x)w^{\sigma}+c(x)w^{\tau}\leq 0 & \hbox{on $B_{R}$}\\
			w\equiv\alpha>0 & \hbox{on $\partial{B_{R}}$}\\
			w>0 & \hbox{on ${B_{R}}$}\,.
			\end{cases}
			\eeq
		Fix any $n\in\N$. Let $\zeta\in\R$ be such that
			\beqs
			\zeta\geq\max\left\{1,\,\frac{n}{\sup_{\partial\Omega}w}\right\}\,,
			\eeqs  
		and define $w_+=\zeta w$. Then, because of \rf{wsuplem2}, the fact that $\Omega\subset\subset B_R$, and the signs of $b(x)$ and $c(x)$, $w_+$ satisfies
			\beqs
			\begin{cases}
			\Delta w_++a(x)w_+- b(x)w_+^{\sigma}+c(x)w_+^{\tau}\leq 0 & \hbox{on $\Omega$}\\
			w_+\geq n & \hbox{on $\partial\Omega$}\\
			w_+>0 & \hbox{on $\Omega$}\,.
			\end{cases}
			\eeqs
		We can suppose that the $R$ chosen above is such that $R\geq R_o$, where $R_o$ is that of the property $\left(\Sigma\right)$ in Definition \ref{defsigma}. This choice implies that there exists a solution $\psi$ of
			\beqs
			\begin{cases}
			\Delta \psi+a(x)\psi- b(x)\psi^{\sigma}+c(x)\psi^{\tau}\geq 0 & \hbox{on $B_{R}$}\\
			\psi\geq 0 & \hbox{on $\overline{B_{R}}$}\,.
			\end{cases}
			\eeqs
		If we define $w_-=\beta\psi$ where $0<\beta\leq 1$, reasoning as above we can find $\beta$ so small that
			\beqs
			\begin{cases}
			\Delta w_- +a(x)w_- - b(x)w_-^{\sigma}+c(x)w_-^{\tau}\geq 0 & \hbox{on $\Omega$}\\
			0\leq w_-\leq w_+ & \hbox{on $\Omega$}\\
			w_-\leq n & \hbox{on $\partial\Omega$}\,.
			\end{cases}
			\eeqs
		Using the monotone iteration scheme we easily arrange a solution $w$ of \rf{nlem1} between $w_-$ and $w_+$. We note that the positivity of $w$ is obvious in the case $\supp c(x)= M$, while in the general case it is a consequence of the strong maximum principle (see \cite{GT, Gi}).
		\end{proof}
We note that the corresponding result for Yamabe-type equations, namely equations of the type \rf{lich} with $c(x)\equiv 0$, can be proved without the additional assumption of property $\left(\Sigma\right)$. Indeed in this case this latter is automatically satisfied by the global subsolution $w_-=0$. The next proposition provides some sufficient conditions for the validity of property $\left(\Sigma\right)$ on $M$.
	\begin{pro}\label{suffsigma}
	Assume the validity of one of the following
		\begin{itemize}
		\item[i)\,\,]	for some $\Lambda>0$
			\beqs
			\lambda_1^{\Delta+a-b+\Lambda c}(M)<0\,;
			\eeqs
		\item[ii)\,] let $C_0=\left\{x\in M: c(x)=0\right\}$ be bounded and such that
			\beqs
			\lambda_1^{\Delta-a}(C_0)>0\,;
			\eeqs
		\item[iii)] there exists a positive subsolution $\varphi_-\in\mathrm{C}^0(M)\cap\mathrm{W}^{1,2}_{loc}(M)$ of \rf{lich}\,.
		\end{itemize}
	Then property $\left(\Sigma\right)$ holds on $M$.
	\end{pro}
		\begin{proof}
		If i) holds true, there exists $R_o>0$ sufficiently large such that
			\beqs
			\lambda_1^{\Delta+a-b+\Lambda c}(B_R)<0\,,
			\eeqs
		for $R\geq R_o$. Accordingly there exists a corresponding positive eigenfunction $\psi$ on $B_{R+\ep}$, $\ep>0$ small say, for which
			\beq\label{lemsigeq}
			\begin{cases}
			\Delta \psi+a(x)\psi- b(x)\psi+\Lambda c(x)\psi\geq 0 & \hbox{on $B_{R}$}\\
			\psi> 0 & \hbox{on $\overline{B_{R}}$}\,.
			\end{cases}
			\eeq
		We let
			\beq\label{lemsigmu}
			0<\mu\leq\min\left\{\Lambda^{\frac{1}{\tau-1}},\left(\sup_{B_R}\psi\right)^{-1}\right\}
			\eeq
		and we define
			\beqs
			\varphi=\mu\psi\,.
			\eeqs
		Note that from \rf{lemsigeq} 
			\beqs
			0\leq\Delta\varphi+a(x)\varphi-b(x)\varphi^{\sigma}\left(\mu\psi\right)^{1-\sigma}+\Lambda c(x)\varphi^{\tau}\left(\mu\psi\right)^{1-\tau}\,.
			\eeqs
		Now, because of \rf{lemsigmu}
			\beqs
			\left(\mu\psi\right)^{1-\sigma}\geq1\quad\quad\hbox{and}\quad\quad\left(\mu\psi\right)^{1-\tau}\leq 1
			\eeqs
		on $\overline{B_R}$. Hence the above inequality implies the validity of \rf{sigpro} with $\varphi$ strictly positive on $\overline{B_R}$.\\
		If ii) holds true, then by Lemma \ref{lem1}, there exists $R_o>0$ sufficiently large such that $C_0\subset B_{R_o}$ and for $R\geq R_o$ there exists a solution $\psi$ of
			\beqs
			\begin{cases}
			\Delta \psi-a(x)\psi-c(x)\psi^{2-\tau}+b(x)\psi^{2-\sigma}\leq0 & \hbox{on $B_R$}\\
			\psi>0 & \hbox{on $\overline{B_R}$}.
			\end{cases}
			\eeqs
		Thus, defining $\varphi=\frac{1}{\psi}$, we have
			\beqs
			\Delta \varphi=-\varphi^2\Delta\psi+2\varphi^3\left|\nabla\psi\right|^2\geq-\varphi^2\Delta\psi\,,
			\eeqs
		which implies \rf{sigpro} always with $\varphi>0$ on $\overline{B_R}$. Case iii) is obvious.
		\end{proof}
%	\begin{rmk}
%	It is obvious that the existence of positive a global subsolution $\varphi_-\in\mathrm{C}^0(M)\cap\mathrm{W}^{1,2}_{loc}(M)$ of \rf{lich} implies the validity of the \Sigma$$-property. What is more interesting is that the validity of condition ii) in Proposition \ref{suffsigma} yields the existence of a global subsolution $\varphi_-\in\mathrm{C}^0(M)\cap\mathrm{W}^{1,2}_{loc}(M)$. To see this
%	\end{rmk}
In the sequel we shall need the following \emph{a priori} estimate. Here $B_T(q)$ denotes the geodesic ball of radius $T$ centered at $q$.
	\begin{lem}\label{aprBT}
	Let $a(x)$, $b(x)$, $c(x)\in\mathrm{C}^{0}(M)$, $\sigma>1$, $\tau<1$, $0<\widetilde{T}<T$, and $\Omega\subset\subset B_{\widetilde{T}}(q)\subset M$. Assume $b(x)>0$ on $\overline{B_T(q)}$. Then there exists an absolute contant $C>0$ such that any positive solution $u\in\mathrm{C}^2(\overline{B_T(q)})$ of
		\beq\label{aprsup}
		\Delta u+a(x)u-b(x)u^{\sigma}+c(x)u^{\tau}\geq 0
		\eeq
	satisfies
		\beq
		\sup_{\Omega}u\leq C\,.
		\eeq
	\end{lem}
		\begin{proof}
		We let $\rho(x)=\dist(x,q)$ and, on the compact ball $\overline{B_T(q)}$ we consider the continous function
			\beqs
			F(x)=\left[T^2-\rho(x)^2\right]^{\frac{2}{\sigma-1}}u(x)
			\eeqs
		where $u(x)$ is any nonnegative $\mathrm{C}^2$ solution of \rf{aprsup}. Note that $\left.F(x)\right|_{\partial B_T(q)}=0$, thus, unless $u\equiv 0$ and in this case there is nothing to prove, $F$ has a positive absolute maximum at some point $\overline{x}\in B_T(q)$. In particular $u(\overline{x})>0$. Now, proceeding as in the proof of Lemma 2.6 in \cite{PRS}, we conclude that, at $\overline{x}$,
			\beqs
			bu^{\sigma-1}\leq\frac{8\left(\sigma+1\right)}{\left(\sigma-1\right)^2}\frac{\rho^2}{\left(T^2-\rho^2\right)^2}+\frac{4}{\sigma-1}\frac{m+\left(m-1\right)A\rho}{T^2-\rho^2} +a_++cu^{\tau-1}\,,
			\eeqs
		for some constant $A\geq0$, independent of $u$. We now state an elementary lemma postponing its proof.
			\begin{lem}\label{lemmunu}
			Let $\alpha$, $\beta\in\left[0,+\infty\right)$, and $\mu$, $\nu\in\left(0,+\infty\right)$. If $t\in\R^+$ satisfies
				\beqs
				t^{\mu}\leq\alpha+\frac{\beta}{t^{\nu}}\,,
				\eeqs
			then
				\beq\label{tmunu}
				t\leq\left(\alpha+\beta^{\frac{\mu}{\mu+\nu}}\right)^{\frac{1}{\mu}}\,.
				\eeq
			\end{lem}
		Since $\sigma>1$ and $\tau<1$, from the lemma we conclude that, at $\overline{x}$,
			\beqs
			u\leq b^{-\frac{1}{\sigma-1}}\left(\frac{8\left(\sigma+1\right)}{\left(\sigma-1\right)^2}\frac{\rho^2}{\left(T^2-\rho^2\right)^2}+\frac{4}{\sigma-1}\frac{m+\left(m-1\right)A\rho}{T^2-\rho^2} +a_++c^{\frac{\sigma-1}{\sigma-\tau}}\right)^{\frac{1}{\sigma-1}}\,.
			\eeqs
		Now the proof proceeds exactly as in Lemma 2.6 of  \cite{PRS} by substituting the $a_+$ there with $a_++c^{\frac{\sigma-1}{\sigma-\tau}}$.
		\end{proof}
		\begin{proof}[Proof of Lemma \ref{lemmunu}]
		If $t^{\mu}\leq\alpha$ we are done, since $\mu>0$ and $\beta\geq0$. In the other case set $s=t^{\mu}$, then  $s>\alpha$ and thus
			\beqs
			s\leq\alpha+\frac{\beta}{s^{\frac{\nu}{\mu}}}<\alpha+\frac{\beta}{\left(s-\alpha\right)^{\frac{\nu}{\mu}}}\,.
			\eeqs
		Setting $r=s-\alpha$ we conclude that
			\beqs
			r^{\frac{\mu+\nu}{\mu}}<\beta\,
			\eeqs	
		and \rf{tmunu} follows.
		\end{proof}
The next simple comparison result reveals quite useful.	
		
	\begin{lem}\label{compf}
	Let  $\Omega\subseteq M$ be a bounded open set. Assume that $f_i:M\times\R\rightarrow\R$ for $i=1,\,2$ are measurable functions such that for all $x\in M$
		\beq\label{monof}
		\frac{f_2(x,s)}{s}\geq\frac{f_1(x,t)}{t}\, ,
		\eeq
	for $s\leq t$. Let $u,v\in\mathrm{C}^0(\overline{\Omega})\cap\mathrm{C}^2(\Omega)$ be solutions on $\Omega$ respectively of
		\beq\label{f1f2}
		\begin{cases}
		\Delta u+f_1(x,u)\geq 0\,;\\
		\Delta v +f_2(x,v)\leq 0\,,
		\end{cases}
		\eeq
	with $u\geq 0$, $v>0$. If $u\leq v$ on $\partial \Omega$, then $u\leq v$ on $\Omega$.
	\end{lem}
		\begin{proof}
		Set $\psi(x)=\frac{u(x)}{v(x)}\in\mathrm{C}^0(\overline{\Omega})\cap\mathrm{C}^2(\Omega)$, from \rf{f1f2} a standard computation yields
			\beq\label{stcomp}
			\Delta\psi\geq\frac{u}{v^2}f_2(x,v)-\frac{1}{v}f_1(x,u)-2\g{\nabla\psi}{\nabla\log v}\,.
			\eeq 
		Now, if we assume by contradiction that $u>v$ somewhere in $\Omega$. Then there exists $\ep>0$ such that
			\beqs
			\Omega_{\ep}=\left\{x\in \Omega\,:\,\psi(x)>1+\ep\right\}\neq\emptyset\,
			\eeqs
		and $\partial\Omega_{\ep}\subset\Omega$. Thus it follows from \rf{stcomp} that the following inequality holds true on $\Omega_{\ep}$
			\beqs
			\Delta\psi+2\g{\nabla\psi}{\nabla\log v} \geq\psi\left[\frac{f_2(x,v)}{v}-\frac{f_1(x,u)}{u}\right]\geq 0\, ,			
			\eeqs 
		moreover $\psi\equiv 1+\ep$ on $\partial\Omega_{\ep}$, and thus by the maximum principle $\psi\leq 1+\ep$ on $\Omega_{\ep}$ contradicting the definition of $\Omega_{\ep}$.
		\end{proof}
	
	\begin{rmk}\label{rmkmonof}
	We note that the hypohteses on $f_i$ of Lemma \ref{compf} are satisfied, for instance, if $f_1=f_2:M\times\R\rightarrow\R$ is a measurable function such that for all $x\in M$
		\beqs
		s\mapsto\frac{f_i(x,s)}{s}
		\eeqs
	is a non increasing function and that the lemma can also be stated for $f_1=f_2:M\times\R^+\rightarrow\R$ with $u$, $v>0$. In particular this is the case for the Lichnerowicz-type nonlinearities considered in this paper, namely
		\beqs
		f(x,s)=a(x)s-b(x)s^{\sigma}+c(x)s^{\tau}\,,
		\eeqs
	with $b(x),\,c(x)$ non negative and $\sigma>1$, $\tau<1$. Indeed, for any fixed $x\in M$ the function
		\beqs
		g_x(s)=\frac{f(x,s)}{s}=a(x)-b(x)s^{\sigma-1}+c(x)s^{\tau-1}
		\eeqs
	is smooth on $\R^+$ and its derivative is given by
		\beqs
		g_x^{\prime}(s)=-(\sigma-1)b(x)s^{\sigma-2}+c(x)(\tau-1)s^{\tau-2}\,,
		\eeqs
	which is non positive by our assumptions on $b(x)$, $c(x)$, $\sigma$, and $\tau$.
	\end{rmk}		
A reasoning similar to that in the proof of Lemma \ref{compf} will be used at the end of the argument in the proof of the next		
		
	\begin{lem}\label{blowup}
	In the assumptions of Lemma \ref{lem2} there exists a positive solution $u$ of the problem
		\beq\label{inftylem}
		\begin{cases}
		\Delta u+a(x)u-b(x)u^{\sigma}+c(x)u^{\tau}=0 & \hbox{on $\Omega$}\\
		u=+\infty & \hbox{on $\partial\Omega$}.
		\end{cases}
		\eeq
	\end{lem}
		\begin{proof}
		For $n\in\N$, let $u_n>0$ on $\Omega$ be the solution of \rf{nlem} obtained in Lemma \ref{lem2} so that
		\beq\label{nlem}
		\begin{cases}
		\Delta u_n+a(x)u_n-b(x)u_n^{\sigma}+c(x)u_n^{\tau}=0 & \hbox{on $\Omega$}\\
		u_n>0 & \hbox{on $\Omega$}\\
		u_n=n & \hbox{on $\partial\Omega$}.
		\end{cases}
		\eeq
		First of all we claim that
			\beq\label{nn+1}
			u_n\leq u_{n+1}\,.
			\eeq
		Indeed, $u_n=n<n+1=u_{n+1}$ on $\partial\Omega$. We then apply Lemma \ref{compf} with the choice $f_1=f_2=f$ and recalling Remark \ref{rmkmonof} we obtain the validity of \rf{nn+1}.\\
		If we show the convergence of the monotone sequence $u_n$ to a function $u$ solution of \rf{inftylem} we are done, indeed $u$ will certainly be positive. Towards this aim, by standard regularity theory, it is enough to show that the sequence $\left\{u_n\right\}$ is uniformly bounded on any compact subset $K$ of $\Omega$. If $K\subset\Omega\setminus B_0$, then we can find a finite covering of balls $\left\{B_i\right\}$ for $K$ such that $b(x)>0$ on each $B_i$. Applying Lemma \ref{aprBT} we deduce the existence of a constant $C_1>0$ such that 
			\beq\label{unC1}
			u_n(x)\leq C_1\quad\quad\forall\,x\in K\,,\,\forall n\in\N\,.
			\eeq
		It remains to find an upper bound on a neighborhood of $B_0$. Towards this end, for $\eta>0$ we let 
			\beqs
			N_{\eta}=\left\{x\in M: d(x,B_0)<\eta\right\}
			\eeqs
		where $\eta$ is small enough that $\overline{N_{\eta}}\subset\Omega$. Furthermore, by the definition of $\eig(B_0)$ and the fact that $\eig(B_0)>0$, we can also suppose to have chosen $\eta$ so small that
			\beqs
			\eig(N_{\eta})>0\,.
			\eeqs
		Now $\partial N_{\eta\slash2}$ is closed and bounded (because $B_0$ is so), hence compact by the completeness of $M$, this implies the existence of a constant $C_2$ such that
			\beqs
			u_n(x)\leq C_2\quad\quad\forall\,x\in \partial N_{\eta\slash 2}\,,\,\forall n\in\N\,;
			\eeqs
		this follows from Lemma \ref{aprBT} by considering a finite covering of $\partial N_{\eta\slash 2}$ with balls of radii less than $\eta\slash2$.\\
		Next we let $\varphi$ be a positive eigenfunction corresponding to $\eig(N_{\eta})$. Then, since $\inf_{N_{\eta\slash2}}\varphi>0$, it follows that there exists a constant $\mu_o>0$ such that
			\beqs
			\mu\varphi(x)> C_2\quad\quad\forall\,x\in \partial N_{\eta\slash 2}\,,\,\forall \mu\geq\mu_o\,.
			\eeqs
		On $N_{\eta\slash2}$ we have
			\beq\label{muphi}
			\Delta\left(\mu\varphi\right)+a(x)\left(\mu\varphi\right)=-\eig(N_{\eta})\left(\mu\varphi\right)<0\,.
			\eeq
		We now choose $\mu\geq\mu_o$ sufficiently large that 
			\beqs
			\mu^{\tau-1}\left(\inf_{N_{\eta\slash2}}\varphi\right)^{\tau-1}\left(\sup_{N_{\eta\slash2}}c(x)\right)<\eig(N_{\eta})\,,
			\eeqs
		this is possible since $\tau<1$ and $\inf_{N_{\eta\slash2}}\varphi>0$. Then, for each $\ep>0$,
			\beq\label{mutau}
			\mu^{\tau-1}\left(\inf_{N_{\eta\slash2}}\varphi\right)^{\tau-1}\left(\sup_{N_{\eta\slash2}}c(x)\right)\left(1+\ep\right)^{\tau-1}<\eig(N_{\eta})\,.
			\eeq
		We let $\psi=\frac{u}{\mu\varphi}$ on $N_{\eta\slash2}$, where $u$ is any of the functions of the sequence $\left\{u_n\right\}$. The same computations as in the proof of Lemma \ref{compf} using \rf{nlem} and \rf{muphi} yields
			\beq\label{Dpsi}
			\Delta\psi+2\g{\nabla\psi}{\nabla\log\left(\mu\varphi\right)}\geq\left(\eig(N_{\eta})+b(x)u^{\sigma-1}-c(x)u^{\tau-1}\right)\psi\,.
			\eeq
		Note that, accordingly to our choice of $\mu$,
			\beqs
			\mu\varphi>C_2>u\quad\quad\hbox{on $\partial N_{\eta\slash2}$}\,.
			\eeqs
		We claim that $\psi\leq1$ on $\partial N_{\eta\slash2}$. By contradiction suppose the contrary. Then, for some $\ep_1>0$, the open set
			\beqs
			\Omega_{\ep_1}=\left\{x\in N_{\eta\slash2} : \psi(x)>1+\ep_1\right\}\neq\emptyset
			\eeqs
		and $\Omega_{\ep_1}\subset\subset N_{\eta\slash2}$. On $\Omega_{\ep_1}$
			\beqs
			u>\left(1+\ep_1\right)\mu\varphi\,.
			\eeqs
		Therefore, since $\tau<1$
			\beqs
			u^{\tau-1}\leq\left(1+\ep_1\right)^{\tau-1}\left(\mu\varphi\right)^{\tau-1}\,.
			\eeqs
		Then, inserting this into \rf{Dpsi}, together with \rf{mutau}, we deduce
			\beqs
			\Delta\psi+2\g{\nabla\psi}{\nabla\log\left(\mu\varphi\right)}\geq\left(\eig(N_{\eta})-\left(\sup_{N_{\eta\slash2}}c(x)\right)\left(1+\ep_1\right)^{\tau-1}\mu^{\tau-1}\varphi^{\tau-1}\right)\psi\geq0\,.
			\eeqs
		By the maximum principle it follows that $\psi$ attains its maximum on $\partial\Omega_{\ep_1}$ but there $\psi(x)=1+\ep_{1}$ contradicting the assumption $\Omega_{\ep_1}\neq\emptyset$.\\
		Thus $\psi\leq1$ on $N_{\eta\slash2}$, that is, $u\leq\mu\psi$ on $N_{\eta\slash2}$. Hence, for all $n\in\N$
			\beqs
			u_n\leq\max\left\{C_2, \sup_{N_{\eta\slash2}}\mu\varphi\right\}\,.
			\eeqs
		This completes the proof of the lemma.
		\end{proof}
		
We are now ready to prove Proposition \ref{thmmax}. The proof, the same of Theorem 6.5 of \cite{MRS}, follows a standard argument and it is reported here for the sake of completeness.
	\begin{proof}[Proof of Proposition \ref{thmmax}]
	First of all we note that, by part iii) of Proposition \ref{suffsigma}, the existence of the global positive subsolution $u_-$ implies that the $\Sigma$-property holds on $M$. We fix an exhausting sequence $\left\{D_k\right\}$ of open, precompact sets with smooth boundaries such that
		\beqs
		B_0\subset D_k\subset\overline{D_k}\subset D_{k+1}\,,
		\eeqs 
	and for each $k$ we denote by $u_k^{\infty}$ the solution of the problem
		\beqs
		\begin{cases}
		\Delta u+a(x)u-b(x)u^{\sigma}+c(x)u^{\tau}=0 & \hbox{on $D_k$}\,;\\
		u=+\infty & \hbox{on $\partial D_k$}\,,
		\end{cases}
		\eeqs
	which exists by Lemma \ref{blowup}. It follows from Lemma \ref{compf} that
		\beq\label{ukuk+1}
		u_-\leq u_{k+1}^{\infty}\leq u_k^{\infty}\quad\quad\hbox{on $\overline{D_k}$}\,.
		\eeq
	Thus $\left\{u_k^{\infty}\right\}$ converges monotonically to a function $u$ solving \rf{lich}, and satisfying, because of \rf{ukuk+1}, $u\geq u_->0$. Let now $u_1>0$ be a second solution of \rf{lich} on $M$. By Lemma \ref{compf}, $u_1\leq u_k^{\infty}$ on $D_k$ for all $k$, and therefore $u_1\leq u$, thus $u$ is a maximal positive solution.
	\end{proof}
We can now prove Theorem A using an existence result for solutions of Yamabe-type equations contained in \cite{MRS}.
	\begin{proof}[Proof of Theorem A]
	By Proposition \ref{thmmax} it follows that to prove the theorem is sufficient to show that assumption \rf{hpthmA} implies the existence of a global subsolution $u_-$ of \rf{lich}. Toward this aim we consider the following Yamabe-type equation
		\beq\label{yam}
		\Delta v+a(x)v-b(x)v^{\sigma}=0\quad\quad\hbox{on $M$}\,,
		\eeq
	with $\sigma$, $a(x)$, and $b(x)$ as in Theorem A. Then, by the sign assumptions \rf{signbc} it follows that a global subsolution $v$ of \rf{yam} is also a global subsolution of \rf{lich}. Now we recall Theorem 6.7 of \cite{MRS} which provides a positive solution $v$ of \rf{yam} under assumptions \rf{Bo>0} and \rf{hpthmA} to conclude the proof. 
	\end{proof}
We conclude the section with the proof of Theorem B. The technique is the same of Theorem A: provide a global subsolution and then apply Proposition \ref{thmmax}. In this case the subsolution is obtained by pasting a subsolution defined inside a compact set and another one defined in the complement of a compact set.
	\begin{proof}[Proof of Theorem B]
	Reasoning as in Lemma \ref{lem1}, assumption \rf{hp1thmB} implies the existence of a solution $\psi\in\mathrm{C}^2(\Omega_2)$ of the following problem
		\beqs
		\begin{cases}
		\Delta \psi+a(x)\psi-b(x)\psi^{\sigma}+c(x)\psi^{\tau}\geq0 & \hbox{on $\Omega_2$}\\
		\psi>0 & \hbox{on $\Omega_2$}\\
		\psi=0 & \hbox{on $\partial\Omega_2$}\,,
		\end{cases}
		\eeqs
	thus $u_1=\psi$ is a subsolution of \rf{lich} in $\Omega_2$. In particular, since $\partial\Omega_1\subset\Omega_2$, if we set $\nu=\min_{\partial\Omega_1}\psi$, we have that $\nu>0$.\\
	Now we note that \rf{hp2thmB} implies that there exists $\mu\in\R^+$ such that
		\beqs
		\sup_{M\setminus\overline{\Omega}_1}\frac{a_-(x)+b(x)}{c(x)}=\mu\,.
		\eeqs
	Let us define $\mu_*=\min\left\{1,\mu^{\frac{1}{\tau-1}},\nu\slash2\right\}$. Then on $M\setminus\overline{\Omega}_1$ we have that
		\beqs
		\begin{aligned}
		\Delta\mu_*+a(x)\mu_*-b(x)\mu_*^{\sigma}+c(x)\mu_*^{\tau} & = a(x)\mu_*-b(x)\mu_*^{\sigma}+c(x)\mu_*^{\tau}\\
		& \geq -a_-(x)\mu_*-b(x)\mu_*+c(x)\mu_*^{\tau}\\
		& = c(x)\mu_*\left[\mu_*^{\tau-1}-\frac{a_-(x)+b(x)}{c(x)}\right]\\
		& \geq c(x)\mu_*\left[\mu_*^{\tau-1}-\mu\right]\\
		& \geq 0\,.
		\end{aligned}
		\eeqs
	Thus $u_2=\mu_*$ is a subsolution of \rf{lich} in $M\setminus\overline{\Omega}_1$. Set
		\beqs
		u_-=
		\begin{cases}
		u_1 & \hbox{on $\overline{\Omega}_1$}\\
		\max\{u_1,u_2\} & \hbox{on $\Omega_2\setminus\overline{\Omega}_1$}\\
		u_2 & \hbox{on $M\setminus\Omega_2$}\,.
		\end{cases}
		\eeqs
	We claim that $u_-$ is the required global subsolution. To prove the claim, we start by noting that the fact that $0<\mu_*<\nu\slash2$ implies $0<u_-\in\mathrm{C}^0(M)\cap\mathrm{W}^{1,2}_{loc}(M)$. For the same reason it is clear that there exists $\ep>0$ such that $u_-$ is a subsolution of \rf{lich} on $\left(\overline{\Omega}_1\right)_{\ep}\cup\left(M\setminus\Omega_2\right)_{\ep}$, where
		\beqs
		\left(U\right)_{\ep}=\bigcup_{x\in U}B_{\ep}(x)
		\eeqs	
	for any set $U\subset M$ ($B_{\ep}(x)$ denotes the geodesic ball of radius $\ep$ centered in $x$). Thus we are left to show that $u_-$ is a subsolution of \rf{lich} on $\Omega_2\setminus\overline{\Omega}_1$, this is a rather standard fact but we sketch the proof here for the sake of completeness. First of all we set
		\beqs
		f(x,v)=a(x)v-b(x)v^{\sigma}+c(x)v^{\tau}\,,
		\eeqs
	then we note that for any test function $\varphi\in\mathrm{W}_0^{1,2}(\Omega_2\setminus\overline{\Omega}_1)$, $\varphi\geq 0$ we have that
		\beqs
		\int_{\Omega_2\setminus\overline{\Omega}_1} \g{\nabla u_1}{\nabla\varphi}-\varphi f(x,u_1)\leq 0\,,
		\eeqs
	and
		\beqs
		f(x,u_2)\geq 0\quad\quad\hbox{on $\Omega_2\setminus\overline{\Omega}_1$}\,.
		\eeqs
	Now, for any $\varphi\in\mathrm{W}_0^{1,2}(\Omega_2\setminus\overline{\Omega}_1)$ and $w\in\mathrm{W}^{1,2}(\Omega_2\setminus\overline{\Omega}_1)$ consider
		\beqs
		H(w,\varphi)=\int_{\Omega_2\setminus\overline{\Omega}_1}\g{\nabla w}{\nabla\varphi}-\varphi f(x,u_-)\,,
		\eeqs
	It is clear that $H(\cdot,\varphi):\mathrm{W}^{1,2}(\Omega_2\setminus\overline{\Omega}_1)\rightarrow\R$ is a continous functional, for any $\varphi$. We want to show that $H(u_-,\varphi)\leq 0$, for any test function $\varphi\geq 0$ on $\Omega_2\setminus\overline{\Omega}_1$.\\ 
	For $\ep>0$, let
		\beqs
		u_{\ep}=\frac{1}{2}\left(u_1+u_2+\sqrt{\left(u_1-u_2\right)^2+\ep^2}\right)\,,
		\eeqs
	then $u_{\ep}$ is smooth with
		\beqs
		\nabla u_{\ep} = \frac{1}{2}\left(1+\frac{u_1-u_2}{\sqrt{\left(u_1-u_2\right)^2+\ep^2}}\right)\nabla u_1\,
		\eeqs	
	moreover, by the definition of $u_-$, $u_{\ep}\stackrel{\,\,\mathrm{W}^{1,2}}{\longrightarrow} u_-$ as $\ep\rightarrow 0$. If $\varphi\in\mathrm{W}_0^{1,2}(\Omega_2\setminus\overline{\Omega}_1)$ and $\ep>0$, then
		\beqs
		\varphi_{\ep}=\frac{1}{2}\left(1+\frac{u_1-u_2}{\sqrt{\left(u_1-u_2\right)^2+\ep^2}}\right)\varphi
		\eeqs
	belongs to $\mathrm{W}_0^{1,2}(\Omega_2\setminus\overline{\Omega}_1)$ and its gradient is given by
		\beqs
		\nabla\varphi_{\ep}= \frac{1}{2}\left(1+\frac{u_1-u_2}{\sqrt{\left(u_1-u_2\right)^2+\ep^2}}\right)\nabla\varphi+ \frac{\ep^2 }{2 \left(\sqrt{\left(u_1-u_2\right)^2+\ep^2}\right)^3}\varphi \nabla u_1\,.
		\eeqs
	The following computation uses the properties of $u_{\ep}$, $\varphi_{\ep}$, and the fact that $u_1$ and $u_2$ are subsolutions
		\beqs
		\begin{aligned}
		H(u_{\ep},\varphi) & =\int_{\Omega_2\setminus\overline{\Omega}_1}\frac{1}{2}\left(1+\frac{u_1-u_2}{\sqrt{\left(u_1-u_2\right)^2+\ep^2}}\right)\g{\nabla u_1}{\nabla\varphi}-\varphi f(x,u_-)\\
		& = \int_{\Omega_2\setminus\overline{\Omega}_1}\g{\nabla u_1}{\nabla \varphi_{\ep}}-\frac{\ep^2 \varphi \left|\nabla u_1\right|^2}{2 \left(\sqrt{\left(u_1-u_2\right)^2+\ep^2}\right)^3}-\varphi f(x,u_-)\\
		& \leq \int_{\Omega_2\setminus\overline{\Omega}_1}\g{\nabla u_1}{\nabla \varphi_{\ep}}-\varphi f(x,u_-)\\
		& \leq \int_{\Omega_2\setminus\overline{\Omega}_1}\varphi_{\ep} f(x,u_1)-\varphi f(x,u_-)\,.
		\end{aligned}
		\eeqs
	Now, since
		\beqs
		\varphi_{\ep}\stackrel{\mathrm{L}^2}{\longrightarrow}
			\begin{cases}
			\varphi & \hbox{if $u_1> u_2$}\\
			0 & \hbox{if $u_1\leq u_2$}
			\end{cases}
		\eeqs
	from the continuity of $H(\cdot,\varphi)$ we conclude that
		\beqs
		H(u_-,\varphi)=\lim_{\ep\rightarrow 0} H(u_{\ep},\varphi)\leq -\int_{\left\{u_1\leq u_2\right\}\cap\Omega_2\setminus\overline{\Omega}_1}\varphi f(x,u_2) \leq 0,
		\eeqs
	for any test function $\varphi$.
	\end{proof}

\section{Uniqueness results and "a priori" estimates}
The aim of this section is to prove uniqueness of positive solutions of equation \rf{lich} on $M$ or outside a relatively compact open set $\Omega$. To avoid technicalities we suppose $u\in\mathrm{C}^2(M)$ or $u\in\mathrm{C}^2(M\setminus\overline{\Omega})$ but this assumption can be relaxed as it will become clear from the arguments we are going to present. Also positivity of $u$ can be relaxed as it will be remarked below. We begin by proving a further comparison result with the aid of the open form of the \emph{q-Weak Maximum Principle} (q-WMP) introduced in \cite{AMiR}, see also \cite{AMaR}. For the present purposes we let $L$ be a linear operator of the form
	\beq\label{opL}
	Lu=\Delta u-\g{X}{\nabla u}
	\eeq
for some vector field $X$ on $M$. Let $q(x)\in\mathrm{C}^0(M)$ be such that $q(x)>0$ on $M$. We recall the following
	\begin{defi}
	We say that the \emph{q-Weak Maximum Principle} holds on $M$ for the operator $L$ in \rf{opL} if, for each $u\in\mathrm{C}^2(M)$ with $u^*=\sup_Mu$ and for each $\gamma\in\R$, with $\gamma < u^*$, we have
		\beq
		\inf_{\Omega_{\gamma}}\left(q(x)Lu\right)\leq0\,,
		\eeq
	where
		\beqs
		\Omega_{\gamma}=\left\{x\in M : u(x)>\gamma\right\}\,.
		\eeqs
	\end{defi}
Next result is contained in Theorem 3.5 in \cite{AMiR}.
	\begin{thm}\label{qwmp}
	The q-WMP holds on $M$ for the operator $L$ if and only if the open q-WMP holds on $M$ for $L$, that is, for each $f\in\mathrm{C}^0(\R)$, for each open set $\Omega\in M$, with $\partial \Omega\neq\emptyset$ and for each $v\in\mathrm{C}^0(\overline{\Omega})\cap\mathrm{C}^2(\Omega)$ satisfying
		\beq
		\begin{cases}
		i)\,\quad q(x)Lv\geq f(v)\quad\hbox{on $\Omega$}\,;\\
		ii)\quad \sup_{\Omega}v<+\infty\,,
		\end{cases}		
		\eeq
	we have that either
		\beq
		\sup_{\Omega}v=\sup_{\partial\Omega}v
		\eeq
	or
		\beq
		f(\sup_{\Omega}v)\leq 0\,.
		\eeq
	\end{thm}
The following is a sufficient condition for the validity of the $q$-WMP (see Theorem 4.1 in \cite{PRS}).
	\begin{thm}\label{suffwmp}
	Let $\M$ be complete and $q(x)\leq C r(x)^{\mu}$ for some constants $C>0$, $2>\mu\geq0$, and $r(x)>>1$. Assume that
		\beq\label{Bvol}
		\liminf_{r\rightarrow+\infty}\frac{\log\vol(B_r)}{r^{2-\mu}}<+\infty\,.
		\eeq
	Then the $q$-WMP holds on $M$ for the operator $\Delta$.
	\end{thm}
We observe that when $q$ is constant or more generally bounded between two positive constants then the $q$-WMP is equivalent to stochastic completeness of the manifold $\M$ this underlines the fact that the $q$-WMP does not require completeness of the metric and that Theorem \ref{suffwmp} indeed gives a sufficient condition.
	\begin{thm}\label{comp}
	Let $a(x), b(x), c(x)\in \mathrm{C}^0(M)$ and $\sigma,\tau\in\R$ be such that $\sigma>1$ and $\tau<1$. Let $\Omega$ be a relatively compact open set in $M$. Assume
	\beq\label{hpcomp}
	\begin{aligned}
	& i) \,\, b(x)>0 \quad\hbox{on $M\setminus\Omega$}\\
	& ii) \,\, c(x)\geq0 \quad\hbox{on $M\setminus\Omega$}\\ 
	& iii) \,\, \sup_M\frac{a_-(x)}{b(x)}<+\infty\\
	& iv)\,\, \sup_M\frac{c(x)}{b(x)}<+\infty
	\end{aligned}	
	\eeq
	where, $a_{-}$ denotes the negative part of $a$. Let $u,v\in \mathrm{C}^0(M\setminus\Omega)\cap\mathrm{C}^2(M\setminus\overline{\Omega})$ be positive solutions of
		\beq\label{subsup}
		\begin{cases}
		L u+a(x)u-b(x)u^{\sigma}+c(x)u^{\tau}\geq 0 & \\
		L v+a(x)v-b(x)v^{\sigma}+c(x)v^{\tau}\leq 0 & 
		\end{cases}	
		\eeq
	on $M\setminus\overline{\Omega}$ satisfying
		\beq\label{vlim}
		\liminf_{x\rightarrow+\infty}v(x)>0\,,
		\eeq
		\beq\label{ulim}
		\limsup_{x\rightarrow+\infty}u(x)<+\infty\,,
		\eeq
	and
		\beq\label{ulv}
		0<u(x)\leq v(x)\quad\quad\hbox{on $\partial\Omega$}\,.
		\eeq 
	Then 
		\beq
		u(x)\leq v(x)
		\eeq 
	on $M\setminus\Omega$ provided that the $1/b$-WMP holds on $M\setminus\overline{\Omega}$ for $L$.
	\end{thm}
	\begin{rmk}
		As it will be observed in the proof of the theorem, in case $0\leq\tau<1$ assumption \rf{hpcomp} iv) can be dropped. 
		\end{rmk}
		\begin{proof}
		Without loss of generality we can suppose that $M\setminus\overline{\Omega}$ is connected. From positivity of $v$, \rf{vlim}, \rf{ulim}, and \rf{ulv} there exist positive constants $C_1$, $C_2$ such that
			\beq\label{uvc}
			v(x)\geq C_1\quad\quad u(x)\leq C_2\quad\quad\hbox{on $M\setminus\overline{\Omega}$}\,.
			\eeq
		We set $\zeta=\sup_{M\setminus\overline{\Omega}}\left(\frac{u}{v}\right)$, from the assumptions on $v$, $u$, and \rf{uvc} it follows that $\zeta$ satisfies
			\beq\label{zlim}
			0<\zeta<+\infty.
			\eeq
		Note that if $\zeta\leq 1$ then $u\leq v$ on $M\setminus\overline{\Omega}$. Thus, assume by contradiction that $\zeta>1$ and define
			\beqs
			\varphi=u-\zeta v\,,
			\eeqs
		then $\varphi\leq 0$ on $M\setminus\overline{\Omega}$. It is not hard to realize, using \rf{zlim} and the definition of $\zeta$, that
			\beq\label{phisup}
			\sup_{M\setminus\overline{\Omega}}\varphi=0.
			\eeq
		We now use \rf{subsup} to compute
			\beq\label{F6}
			\begin{aligned}
			L\varphi & \geq -a(x)\varphi+b(x)\left[u^{\sigma}-(\zeta v)^{\sigma}\right]-c(x)\left[u^{\tau}-(\zeta v)^{\tau}\right]\\
			& \quad\quad\quad\quad+b(x)\zeta v\left[(\zeta v)^{\sigma-1}-v^{\sigma-1}\right]+c(x)\zeta v\left[v^{\tau-1}-(\zeta v)^{\tau-1}\right]\,.
			\end{aligned}
			\eeq
		We let
			\beqs
			h(x)=\left\{\begin{array}{cc}
			\sigma u^{\sigma-1}(x) & \mathrm{if} \quad u(x)=\zeta v(x) \\
			{} & \\
			\displaystyle{\frac{\sigma}{u(x)-\zeta v(x)}\int_{\zeta v(x)}^{u(x)}t^{\sigma-1}dt} & \mathrm{if} \quad u(x)<\zeta v(x).
			\end{array} \right.
			\eeqs
		and, similarly, for $\tau\neq0$,
			\beqs
			j(x)=\left\{\begin{array}{cc}
			-\tau u^{\tau-1}(x) & \mathrm{if} \quad u(x)=\zeta v(x) \\
			{} & \\
			\displaystyle{\frac{\tau}{\zeta v(x)-u(x)}\int_{\zeta v(x)}^{u(x)}t^{\tau-1}dt} & \mathrm{if} \quad u(x)<\zeta v(x).
			\end{array} \right.
			\eeqs
		In case $\tau=0$ choose $j(x)\equiv0$. Observe that $h$ and $j$ are continous on $M\setminus\overline{\Omega}$ and $h$ is non-negative. Using $h$ and $j$, and observing that $-a(x)\varphi\geq a_{-}(x)\varphi$, from \rf{F6} we obtain
			\beq\label{F7}
			\begin{aligned}
			L\varphi & \geq \left[a_-(x)+b(x)h(x)+c(x)j(x)\right]\varphi\\
			&\quad\quad\quad\quad +b(x)\zeta v\left[(\zeta v)^{\sigma-1}-v^{\sigma-1}\right]+c(x)\zeta v\left[v^{\tau-1}-(\zeta v)^{\tau-1}\right]\,.
			\end{aligned}
			\eeq
		Let
			\beqs
			\Omega_{-1}=\{x\in M\setminus\overline{\Omega} : \varphi(x)>-1 \}.
			\eeqs
		Since $u$ is bounded above on $M\setminus\overline{\Omega}$, there exists a constant $C>0$ such that
			\beq\label{F8}
			v(x)=\frac{1}{\zeta}(u(x)-\varphi(x))\leq\frac{1}{\zeta}(C+1)
			\eeq
		on $\Omega_{-1}$. Using the definitions of $h$ and $j$, from the mean value theorem for integrals, we deduce
			\beqs
			h(x)=\sigma y_h^{\sigma-1}, \quad\quad j(x)=-\tau y_j^{\tau-1}
			\eeqs
		for some $y_{h}=y_{h}(x)$ and $y_j=y_j(x)$ in the range $[u(x),\zeta v(x)]$. Since $u(x)$ and $v(x)$ are bounded above on $\Omega_{-1}$
			\beq\label{F9}
			\max\left\{h(x),j(x\right\})\leq C
			\eeq
		on $\Omega_{-1}$ for some constant $C>0$.
		Next we recall that $b(x)>0$ on $M\setminus\overline{\Omega}$ to rewrite \rf{F7} in the form
			\beqs
			\begin{aligned}
			\frac{1}{b(x)}L\varphi & \geq \left[\frac{a_-(x)}{b(x)}+h(x)+\frac{c(x)}{b(x)}j_+(x)\right]\varphi\\
			&\quad\quad\quad\quad +\zeta v\left[(\zeta v)^{\sigma-1}-v^{\sigma-1}\right]+\frac{c(x)}{b(x)}\zeta v\left[v^{\tau-1}-(\zeta v)^{\tau-1}\right]\,.\\
			\end{aligned}
			\eeqs
		Since $\varphi\leq 0$, \rf{hpcomp} and \rf{F9} imply
			\beqs
			\left[\frac{a_-(x)}{b(x)}+h(x)+\frac{c(x)}{b(x)}j_+(x)\right]\varphi\geq C\varphi
			\eeqs
		for some constant $C>0$ on $\Omega_{-1}$. For further use we observe here that when $0\leq\tau<1$, $j_+(x)\equiv 0$ so that in this case assumption \rf{hpcomp} iv) is not needed to obtain this last inequality. Thus
			\beqs
			\frac{1}{b(x)}L\varphi\geq C\varphi+\zeta v\left[(\zeta v)^{\sigma-1}-v^{\sigma-1}\right]+\frac{c(x)}{b(x)}\zeta v\left[v^{\tau-1}-(\zeta v)^{\tau-1}\right]
			\eeqs
		on $\Omega_{-1}$. Recalling the elementary inequalities 
			\beq\label{hlp}
			\begin{cases}
			a^s-b^s \geq sb^{s-1}(a-b) & \hbox{for $s<0$ and $s>1$}\,;\\
			a^s-b^s \geq sa^{s-1}(a-b) & \hbox{for $0\leq s\leq 1$}\,,
			\end{cases}
			\eeq
		$a$, $b\in\R^+$, coming from the mean value theorem for integrals (see Theorem 41 in \cite{HLP}) we conclude
			\beqs
			\frac{1}{b(x)}L\varphi\geq C\varphi+(\sigma-1)\zeta^{\min\left\{1,\sigma-1\right\}}(\zeta-1)v^{\sigma}+\left(1-\tau\right)\frac{c(x)}{b(x)}\frac{\zeta-1}{\zeta^{1-\tau}}v^{\tau}\,,
			\eeqs
		on $\Omega_{-1}$. Now we use the fact that $\tau<1$, $v$ is bounded from below by a positive constant, \rf{hpcomp} i), ii), iv) to get (again if $0\leq\tau<1$ we do not need \rf{hpcomp} iv))
			\beqs
			\frac{1}{b(x)}L\varphi\geq C\varphi+B \quad\hbox{on  $\Omega_{-1}$},
			\eeqs
		for some positive constants $B$, $C$. Finally, we choose $0<\varepsilon<1$ sufficiently small that
			\beqs
			C\varphi>-\frac{1}{2}B
			\eeqs
		on
			\beqs
			\Omega_{-\varepsilon}=\{x\in M\setminus\overline{\Omega} : \varphi(x)>-\varepsilon \}\subset\Omega_{-1}.
			\eeqs
		Therefore
			\beq\label{omep}
			\frac{1}{b(x)}L\varphi\geq \frac{1}{2}B>0 \quad on \quad \Omega_{-\varepsilon}.
			\eeq
		Furthermore, note that
			\beqs
			\varphi(x)\leq\min\left\{-\ep,\ (1-\zeta)\min_{\partial\Omega}v\right\}<0\quad\quad\hbox{on $\Omega_{-\ep}$}\,.
			\eeqs
		As a consequence $\sup_{\partial\Omega_{-\ep}}\varphi<0$ while $\sup_{\Omega_{-\ep}}\varphi=0$. By Theorem \ref{qwmp}, \rf{omep} and the above fact, we obtain the required contradiction, proving
that $\zeta\leq 1$.
		\end{proof}
As an immediate consequence of Theorem \ref{comp} we obtain the following uniqueness result
	\begin{cor}\label{uniq}
	In the assumptions of Theorem \ref{comp} the equation
		\beqs
		L u+a(x)u-b(x)u^{\sigma}+c(x)u^{\tau}= 0 \quad \hbox{on $M\setminus\overline{\Omega}$,}	
		\eeqs
	admits at most a unique solution $u\in\mathrm{C}^0(M\setminus\Omega)\cap\mathrm{C}^2(M\setminus\overline{\Omega})$ with assigned boundary data on $\partial\Omega$ and satisfying
		\beq\label{CuC}
		C_1\leq u(x)\leq C_2\quad\quad\hbox{on $M\setminus\overline{\Omega}$}
		\eeq	
	for some positive constants $C_1$, $C_2$, provided that the $1\slash b$-WMP holds on $M$ for the operator $L$.
	\end{cor}
We observe that the two main assumptions in Corollary \ref{uniq} are the validity of the $1\slash b$-WMP on $M$ for $L$ and the validity of the bounds \rf{CuC}. In case $L=\Delta$ in Theorem \ref{suffwmp} we have given a sufficient condition for the validity of the $1\slash b$-WMP, for $b>0$ everywhere it remains to analyze \rf{CuC}. Towards this aim we recall the following result companion of Theorem \ref{suffwmp} and whose proof can be easily adapted from \cite{PRS} and \cite{PRSvol}.
	\begin{thm}\label{thmb}
	Let $\M$ be a complete Riemannian manifold and $a(x)$, $b(x)\in \mathrm{C}^0(M)$. Assume that $\left\|a_+\right\|_{\infty}< +\infty$, $b(x)>0$ on $M$, and
		\beq\label{brmu}
		b(x)\geq \frac{C}{r(x)^{\mu}}
		\eeq	
	outside a compact set $K$ for some constants $C>0$ and $\mu<2$. Assume \rf{Bvol} and
		\beqs
		\sup_{M}\frac{a_+(x)}{b(x)}<+\infty\,. 
		\eeqs
	Let $u \in \mathrm{C}^2(M)$ be a non-negative solution of
		\beqs
		\Delta u + a(x)u - b(x)u^{\sigma} \geq0 \quad\quad \hbox{on $\Omega_{\gamma}$}\,,
		\eeqs
	where $\sigma>1$ and
		\beqs
		\Omega_{\gamma}=\left\{x\in M : u(x)>\gamma\right\}
		\eeqs
	for some $\gamma < u^*\leq+\infty$. Then $u^*<+\infty$. Furthermore, having set
		\beqs
		H_{\gamma}=\sup_{\Omega_{\gamma}}\frac{a_+(x)}{b(x)}\,,
		\eeqs
	we have
		\beq\label{u*b}
		u^*\leq H_{\gamma}^{{1}\slash({\sigma-1})}\,.
		\eeq
	\end{thm}
We are now ready to prove
	\begin{pro}\label{upapr}
	Let $\M$ be a complete Riemannian manifold. Let $a(x)$, $b(x)$, $c(x)\in \mathrm{C}^0(M)$, and assume $\left\|a_++c_+\right\|_{\infty}< +\infty$, that $b(x)>0$ on $M$ and that it satisfies \rf{brmu} for some $\mu<2$ outside a compact set. Suppose the validity of \rf{Bvol} and of
		\beq\label{acbapr}
		\sup_M\frac{a_+(x)+c_+(x)}{b(x)}<+\infty\,. 
		\eeq
	Let $\sigma > 1$, $\tau<1$, and $u \in\mathrm{C}^2(M)$ be a positive solution of
		\beq\label{apreq}
			\Delta u+ a(x)u - b(x)u^{\sigma} + c(x)u^{\tau}\geq 0
		\eeq
	on 
		\beq\label{omgam}
		\Omega_{\gamma}=\left\{x\in M : u(x)>\gamma\right\}\,,
		\eeq
	for some $\gamma < u^*\leq+\infty$.	Then $u^*<+\infty$ and indeed
		\beq\label{u*H}
		u^*\leq\max\left\{\gamma^*,\,H_{\gamma^*}^{1\slash(\sigma-1)}\right\}
		\eeq	
	where $\gamma^*=\max\left\{1,\gamma\right\}$ and 
		\beqs
		H_{\gamma^*}=\sup_{\Omega_{\gamma^*}}\frac{a_+(x)+c_+(x)}{b(x)}\,.
		\eeqs
	\end{pro}
		\begin{proof}
		First we show that $u^*<+\infty$. We can suppose $u^*>1$. If $\gamma<1$ we let $\widetilde{\gamma}$	be such that $1\leq\widetilde{\gamma}<u^*$ and note that $\Omega_{\widetilde{\gamma}}\subset\Omega_{\gamma}$. It follows that \rf{apreq} holds on $\Omega_{\widetilde{\gamma}}$. Thus, without loss of generality, we can suppose $\gamma\geq1$. Since $u^{\tau}\leq u$ on $\Omega_{\gamma}$, from \rf{apreq} we have
			\beqs
			\begin{aligned}
			\Delta u+ a_+(x)u - b(x)u^{\sigma} + c_+(x)u \geq & \Delta u+ a_+(x)u - b(x)u^{\sigma} + c_+(x)u^{\tau}\\
			\geq & \Delta u+ a(x)u - b(x)u^{\sigma} + c(x)u^{\tau}\\
			\geq & 0\,,
			\end{aligned}
			\eeqs
		on $\Omega_{\gamma}$; in other words
			\beqs
			\Delta u+ \left[a_+(x)+c_+(x)\right]u - b(x)u^{\sigma} \geq 0 \quad \hbox{on $\Omega_{\gamma}$}.
			\eeqs
		Applying Theorem \ref{thmb} we deduce $u^*<+\infty$. To prove \rf{u*H} first let $\gamma\geq1$ so that $\gamma^*=\gamma$, $\Omega_{\gamma^*}=\Omega_{\gamma}$, $H_{\gamma^*}=H_{\gamma}$ and \rf{u*H} follows directly from \rf{u*b} of Theorem \ref{thmb}. Suppose now $\gamma<1$. Then $\gamma^*=1$ and $\Omega_{\gamma^*}=\Omega_1\subset\Omega_{\gamma}$. If $\Omega_{1}=\emptyset$ then $u^*\leq\gamma^*$. If $\Omega_{1}\neq\emptyset$ then \rf{apreq} holds on $\Omega_{1}$ and applying again Theorem \ref{thmb} we deduce the validity of \rf{u*H}.
		\end{proof}

Now, as in case ii) of Proposition \ref{suffsigma}, we are going to exploit the simmetry of equation \rf{lich} to obtain a bilateral \emph{a priori} estimate. This is the content of the next crucial 

	\begin{thm}\label{thmapr}
	Let $\M$ be a complete Riemannian manifold. Let $a(x)$, $b(x)$, $c(x)\in \mathrm{C}^0(M)$, $\left\|a_++c\right\|_{\infty}<+\infty$, $\left\|a_-+b\right\|_{\infty}<+\infty$. Moreover assume that $b(x)>0$ and $c(x)>0$ on $M$, and that both satisfy \rf{brmu} for some $\mu<2$. Suppose the validity of \rf{Bvol} and of 
		\beq\label{acb}
		\sup_{M}\frac{a_+(x)+c(x)}{b(x)}= H<+\infty \,,
		\eeq
	and
		\beq\label{abc}
		\sup_{M}\frac{a_-(x)+b(x)}{c(x)}= K<+\infty \,.
		\eeq
	Let $\sigma > 1$, $\tau<1$. Then any positive, $\mathrm{C}^2$ solution of
		\beq\label{apreq2}
			\Delta u+ a(x)u - b(x)u^{\sigma} + c(x)u^{\tau}=0 \quad \hbox{on $M$},
		\eeq	
	satisfies 
		\beq\label{aprsubsup}
		\mathcal{K} \leq u(x) \leq \mathcal{H} \quad\hbox{on $M$},
		\eeq
	where
		\beq\label{calKH}
		\mathcal{K}= \min\left\{1,\,K^{1\slash(\tau-1)}\right\},\quad\quad\mathcal{H}=\max\left\{1,\,H^{1\slash(\sigma-1)}\right\}.
		\eeq
	\end{thm}
		\begin{proof}
		Suppose $\Omega_1=\left\{x\in M : u(x)>1\right\}\neq\emptyset$, then the validity of (\ref{apreq2}) implies that of
			\beqs
			\Delta u+ a(x)u - b(x)u^{\sigma} + c(x)u^{\tau}\geq 0 \quad \hbox{on $\Omega_1$},
			\eeqs 
		thus the estimate from above in \rf{aprsubsup} follows from Proposition \ref{upapr}. In case $\Omega_1=\emptyset$ the same estimate is trivially true because of the definition \rf{calKH} of $\mathcal{H}$. For the estimate from below we consider the function $v=\frac{1}{u}\in\mathrm{C}^2(M)$, since $u>0$ on $M$. Since $\Delta v=-v^2\Delta u+2v^3\left|\nabla u\right|^2$, using \rf{apreq2} we have
			\beqs
			\Delta v +\widetilde{a}(x)v-\widetilde{b}(x)v^{\widetilde{\sigma}}+\widetilde{c}(x)v^{\widetilde{\tau}}\geq 0\quad\hbox{on $M$,}
			\eeqs 
		where we have set $\widetilde{a}(x)=-a(x)$, $\widetilde{b}(x)=c(x)$, $\widetilde{c}(x)=b(x)$, $\widetilde{\sigma}=2-\tau>1$, and $\widetilde{\tau}=2-\sigma<1$. Now, since
			\beqs
			\frac{a_-(x)+b_+(x)}{c(x)}=\frac{\widetilde{a}_+(x)+\widetilde{c}_+(x)}{\widetilde{b}(x)}\,,
			\eeqs
		we can reason as above and deduce
			\beqs
			v\leq\max\left\{1,K^{1\slash(\widetilde{\sigma}-1)}\right\}=\max\left\{1,K^{1\slash(1-\tau)}\right\}\,
			\eeqs
		and the lower bound in \rf{aprsubsup} follows immediately from the definition of $v$.
		\end{proof}

We note that the existence of solutions for equation \rf{lich} can be easily obtained under the hypoteses of Theorem \ref{thmapr} by direct application of the \emph{monotone iteration scheme} of H. Amann (see for instance \cite{Sa} or \cite{MRS}). Indeed in this case it is relatively easy to find an ordered pair of global sub and supersolutions. The key point is that their existence is a consequence of the \emph{a priori} estimate. This is the content of the following
	
	\begin{lem}\label{subsuper}
	Let $\M$ be a complete Riemannian manifold. Let $a(x)$, $b(x)$, $c(x)$, $\sigma$, $\tau$, $H$, $K$, $\mathcal{H}$, and $\mathcal{K}$ be as in Theorem \ref{thmapr}. Then $u^+\equiv\mathcal{H}$ and $u^-\equiv\mathcal{K}$ are respectively a global supersolution and a global subsolution of \rf{apreq2}. Moreover $u^-\leq u^+$.
	\end{lem}
		\begin{proof}
		First of all we note that since $\mathcal{H}\geq 1$ and $\tau<1$, then it follows that $\mathcal{H}^{\tau-1}\leq 1$. This implies that
			\beqs
			\begin{aligned}
			\Delta u^++ a(x)u^+ - b(x)(u^+)^{\sigma} + c(x)(u^+)^{\tau} & = \mathcal{H}\left[a(x) - b(x)\mathcal{H}^{\sigma-1} + c(x)\mathcal{H}^{\tau-1}\right]\\
			& \leq\frac{\mathcal{H}}{b(x)}\left[\frac{a_+(x)+c(x)}{b(x)}-\mathcal{H}^{\sigma-1}\right]\\
			& \leq 0
			\end{aligned}
			\eeqs
		where in the last passage we have used \rf{acb} and the fact that $\mathcal{H}\geq H^{\frac{1}{\sigma-1}}$, thus $u^+$ is a global supersolution. The proof of the fact that $u^-$ is a subsolution is analogous and $u^-\leq u^+$ follows from the definitions of $\mathcal{H}$ and $\mathcal{K}$.
		\end{proof}

From this we immediately deduce the next existence result (see also \cite{RRV} for a similar result). 

	\begin{thm}\label{posex}
	Let $\M$ be a complete Riemannian manifold. Let $a(x)$, $b(x)$, $c(x)$, $\sigma$, $\tau$ be as in Theorem \ref{thmapr} and assume that $a(x)$, $b(x)$, $c(x)\in\mathrm{C}^{0,\alpha}(M)$ for some $\alpha>0$. Then \rf{lich} has a positive solution $u\in\mathrm{C}^2(M)$.
	\end{thm}
		\begin{proof}
		Let $\left\{\Omega_k\right\}_{k\in\N}$ be a family of bounded open sets with smooth boundaries such that
			\beqs
			\begin{aligned}
			\Omega_k\subset\subset\Omega_{k+1}\,;\\			
			\bigcup_{k\in\N}\Omega_k=M\,.
			\end{aligned}
			\eeqs
	For each $k\in\N$ consider the Dirichlet problem
		\beq\label{dirk}
		\begin{cases}
		\Delta v+ a(x)v - b(x)v^{\sigma} + c(x)v^{\tau}=0 & \hbox{on $\Omega_k$}\,;\\
		v=u^+ & \hbox{on $\partial\Omega_k$}\,,
		\end{cases}
		\eeq
	where $u^+=\mathcal{H}$ is the global supersolution of Lemma \ref{subsuper}. Since $u^+$ and $u^-$ of Lemma \ref{subsuper} are respectively a supersolution and a subsolution of \rf{dirk} for any $k\in\N$, it follows from the {monotone iteration scheme} (see for instance Theorem 2.1 of \cite{Sa}) that for any $k$ there exists a solution $v_k\in\mathrm{C}^{2,\alpha}(\Omega_k)$ of \rf{dirk} such that $u^-\leq v_k\leq u^+$. From Lemma \ref{compf} it follows that
		\beqs
		u^-\leq v_{i}\leq v_{j}\leq u^+\quad\hbox{on $\Omega_k$}
		\eeqs
	for all $i,j\in\N$ such that $i\geq j\geq k$. Thus, from the Schauder interior estimates and the compactness of the embedding $\mathrm{C}^{2,\alpha}(\Omega_k)\subset\mathrm{C}^{2}(\Omega_k)$ it follows that the $v_k$ converge uniformly on compact sets to a solution $u\in\mathrm{C}^{2}(M)$ of \rf{lich}. Moreover $u(x)\geq u^->0$.
	\end{proof}

Putting together Theorem \ref{suffwmp}, Corollary \ref{uniq} with $\Omega=\emptyset$, Theorem \ref{thmapr}, and Theorem \ref{posex} we have the following
	\begin{cor}\label{corLio}
	In the assumptions of Theorem \ref{thmapr} with $0\leq\mu<2$ the equation
		\beqs
			\Delta u+ a(x)u - b(x)u^{\sigma} + c(x)u^{\tau}=0 \quad \hbox{on $M$}. 
		\eeqs
	admits a unique positive solution $u\in\mathrm{C}^2(M)$.
	\end{cor}
	
The next corollary deals with the easier case where $a(x)$, $b(x)$, and $c(x)$ are of the form $\zeta f(x)$ where $0<f(x)\in\mathrm{C}^0(M)\cap\mathrm{L}^{\infty}(M)$ and $\zeta\in\R$. It generalizes Theorem 2 of \cite{Ma} and Theorem 7 of \cite{MW}. Furthermore it should be compared with Theorem 3.15 and Example 3.18 of \cite{DAM}.
	
	\begin{cor}\label{corapr}
	Let $\M$ be a complete Riemannian manifold. Let $\alpha$, $\beta$, $\gamma\in\R$ such that $\beta$, $\gamma>0$. Let $\sigma > 1$, $\tau<0$, $0<f(x)\in\mathrm{C}^0(M)\cap\mathrm{L}^{\infty}(M)$  satisfying \rf{brmu} outside a compact set for some $\mu<2$, and assume the validity of \rf{Bvol}. Then the unique positive solution of
		\beqs
			\Delta u + f(x)\left(\alpha u - \beta u^{\sigma} + \gamma u^{\tau}\right)=0 \quad \hbox{on $M$}
		\eeqs
	is given by $u\equiv\lambda$, where $\lambda\in\R^+$ satisfies $p(\lambda)=0$, with
		\beqs
		p(t)=\alpha +\beta t^{\sigma-1}-\gamma t^{\tau-1}\,.
		\eeqs
	\end{cor}

We conclude the section with a second uniqueness result whose proof is based on that of Theorem 4.1 of \cite{BRS}, see also Theorem 5.1 in \cite{MRS}.
	\begin{thm}
	Let $\M$ be a complete manifold, $a(x), b(x), c(x)\in \mathrm{C}^0(M)$, $\sigma>1$, $\tau<1$, and assume \rf{signbc}, that is,
		\beqs
		b(x)\geq0,\quad\quad c(x)\geq0\,,
		\eeqs
	and
		\beq\label{bc0}
		b(x)+c(x)\not\equiv 0 \quad\hbox{on $M$}\,.
		\eeq
	Let $u$, $v\in\mathrm{C}^2(M)$ be positive solutions of
		\beqs
		\Delta u+a(x)u-b(x)u^{\sigma}+c(x)u^{\tau}= 0 \quad \hbox{on $M$,}	
		\eeqs
	such that
		\beq\label{uvint}
		\left\{\int_{\partial B_r}\left(u-v\right)^2\right\}^{-1}\notin\mathrm{L}^1(+\infty)\,.
		\eeq
	Then $u\equiv v$ on $M$.
	\end{thm}	
	\begin{rmk}
	Note that condition \rf{uvint} is implied by $u-v\in\mathrm{L}^2(M)$ or even by the weaker request
		\beqs
		\int_{B_r}\left(u-v\right)^2=o(r^2)\quad\quad\hbox{as $r\rightarrow+\infty$}\,.
		\eeqs 
	See for instance Proposition 1.3 in \cite{RS}.
	\end{rmk}
		\begin{proof}
		The proof follows, mutatis mutandis, that reported in Theorem 5.1 of \cite{MRS} up to equation (5.7) that now becomes
			\beq\label{intbc}
			\begin{aligned}
			& \int_{M}b(x)\left(v^2-u^2\right)\left(v^{\sigma-1}-u^{\sigma-1}\right)+\int_{M}c(x)\left(v^2-u^2\right)\left(u^{\tau-1}-v^{\tau-1}\right)+\\
			& \quad\quad+\int_{M}\left\{\left|\nabla v-\frac{v}{u}\nabla u\right|^2+\left|\nabla u-\frac{u}{v}\nabla v\right|^2\right\}=0\,.
			\end{aligned}
			\eeq
		Because of \rf{signbc} we deduce $\left|\nabla u-\frac{u}{v}\nabla v\right|\equiv 0$ on $M$ so that $u=Av$ for some constant $A>0$. Substituting into \rf{intbc} yields
			\beqs
			\left(1-A^2\right)\left(1-A^{\sigma-1}\right)\int_Mb(x)v^{\sigma+1}=0
			\eeqs
		and
			\beqs
			\left(1-A^2\right)\left(A^{\tau-1}-1\right)\int_Mc(x)v^{\tau+1}=0\,.
			\eeqs
		Since $v>0$, \rf{bc0} implies $A=1$, that is, $u=v$ on $M$.
		\end{proof}
		\begin{rmk}
		The exponent $2$ in \rf{uvint} is sharp, see the discussion after Theorem 5.1 in \cite{MRS}.
		\end{rmk}

\section{A further comparison and uniqueness result}
In this section we prove a comparison result and a corresponding uniqueness result based on a spectral property of the operator $L=\Delta+a(x)$. As we have seen, the request $\eig(M)<0$ facilitates the search of solutions of equation \rf{lich}. Somehow the opposite request seems to limitate the existence of solutions.\\

We recall that $L$ has finite index if and only if there exists a positive solution $u$ of the differential inequality
	\beq\label{uind}
	Lu\leq 0\,,
	\eeq
outside a compact set $K$. In what follows we shall denote with $\MG$ a triple with the following properties: $\M$ is a complete manifold with a preferred origin $o$ and $G\in\mathrm{C}^2(\Mo)$, $G: \Mo\rightarrow\R^+$ is such that
	\beq\label{Ghp}
	\begin{cases}	
	\textit{i)}\,\,\,\, \Delta G\leq0 & \hbox{on $\Mo$}\,;\\
	\textit{ii)}\,\,\, G(x)\rightarrow+\infty & \hbox{as $x\rightarrow o$}\,;\\
	\textit{iii)}\,\, G(x)\rightarrow 0 & \hbox{as $x\rightarrow +\infty$}\,,
	\end{cases}
	\eeq
Clearly a good candidate for $G$ is the (positive) Green kernel at $o$ on a non-parabolic complete manifold, which, however, might not satisfy \rf{Ghp} iii). Observe that, for instance by the work of Li and Yau, \cite{LY}, iii) is satisfied by the Green kernel if $\Ric\geq0$. Other examples always concerning the Green kernel, are given by non-parabolic complete manifolds supporting a Sobolev inequality of the type
	\beq\label{sobolev}
	S(\alpha)^{-1}\left(\int_Mv^{\frac{2}{1-\alpha}}\right)^{1-\alpha} \leq \int_M\left|\nabla v\right|^2\quad\quad\hbox{for all $v\in\mathrm{C}_c^{\infty}(M)$}
	\eeq
for some $\alpha\in(0,1)$, $S(\alpha)>0$, and for all $v\in\mathrm{C}^{\infty}_c(M)$. For further examples see \cite{Ni} and the references therein. Note that, in these results, the authors also describe the behavior of $G(x)$ at infinity from above and below. That of $\left|\nabla G(x)\right|$ from above can often be obtained by classical gradient estimates. This is helpful for instance in Theorem \ref{uniq2} below.\\
However, since we only require superharmonicity of $G$, under a curvature assumption we can use \emph{transplantation} from a non-parabolic model. The argument is as follow. Assume $\M$ is a $m$-dimensional manifold with a pole $o$ and with radial sectional curvature (with respect to $o$) $\mathrm{K}_{rad}$ satisfying
	\beqs
	\mathrm{K}_{rad}\leq-F(r(x))\quad\quad\hbox{on $M$,}
	\eeqs
with $r(x)=\dist_M(x,o)$, $F\in\mathrm{C}^0(\R_0^{+})$. Let $g$ be a $\mathrm{C}^2$-solution of the problem
	\beq\label{riccati}
	\begin{cases}
	g''-F(r)g\leq0\\
	g(0)=0\,,\quad g'(0)=1\,,
	\end{cases}
	\eeq
and suppose that $g>0$ on $\R^+$. Note that this request is easily achieved by bounding appropriately $F$ from above. See for instance \cite{BMR}. Then by the Laplacian comparison theorem
	\beq\label{lapcomthm}
	\Delta r\geq (m-1)\frac{g'(r)}{g(r)}\quad\quad\hbox{on $M\setminus\left\{o\right\}$,}
	\eeq
and weakly on $M$. Consider the $\mathrm{C}^2$-model $M_g$ defined by $g$ with the metric
	\beqs
	\g{\,}{\,}_g=dr^2+g^2(r)d\theta^2
	\eeqs
on $M\setminus\left\{o\right\}=\R^+\times\mathrm{S}^{m-1}$, $\mathrm{S}^{m-1}$ the unit sphere, $d\theta^2$ its canonical metric. Then $M\setminus\left\{o\right\}$ is non-parabolic if and only if $\frac{1}{g^{m-1}}\in\mathrm{L}^1(+\infty)$. Now we transplant the positive Green function on $M\setminus\left\{o\right\}$ evaluated at $(y,o)$ to $M$, that is, we let
	\beqs
	G(x)=\int_{r(x)}^{+\infty}\frac{ds}{g(s)^{m-1}}>0\quad\quad\hbox{on $M\setminus\left\{o\right\}$.}
	\eeqs
An immediate computation yields
	\beqs
	\Delta G(x)=-\frac{1}{g(r(x))^{m-1}}\left\{\Delta r(x)-(m-1)\frac{g'(r(x))}{g(r(x))}\right\}\quad\quad\hbox{on $M\setminus\left\{o\right\}$.}
	\eeqs
Hence, (\ref{lapcomthm}) implies (\ref{Ghp}) i). The remaining of (\ref{Ghp}) be trivially satisfied.\\
Thus we solve the problem by looking for a solution of (\ref{riccati}) satisfying $\frac{1}{g^{m-1}}\in\mathrm{L}^1(+\infty)$. For a detailed analysis we refer also to \cite{BMR1}.
On $\Mo$ we define
	\beq
	t(x)=-\frac{1}{2}\log G(x)
	\eeq
and, for $s\in\R$, we set 
	\beq
	\Lambda_s=\left\{x\in\Mo : t(x)<s\right\}\cup\left\{o\right\}\,,
	\eeq
so that
	\beq
	\partial\Lambda_s=\left\{x\in\Mo : t(x)=s\right\}\,.
	\eeq
Note that, because of \rf{Ghp} ii), $\Lambda_s$ is open and $\left\{\Lambda_s\right\}_{s\in\R}$ is an exhausting family of open sets. Property \rf{Ghp} iii) and completeness of $\M$ implies that $\overline{\Lambda_s}$ is compact for each $s\in\R$.\\

We are now ready to prove the following
	\begin{thm}\label{thmind}
	Let $\MG$ be as above and suppose that $a(x)\in\mathrm{C}^0(M)$ satisfy
		\beq\label{aind}
		a(x)\leq\left\{1+\frac{1}{\log^2G(x)}\left[1+\frac{1}{\log^2\left(-\log\sqrt{G(x)}\right)}\right]\right\}\frac{\left|\nabla\log G(x)\right|^2}{4}\,,
		\eeq
	outside a compact set $K$. Then the operator $L=\Delta+a(x)$ has finite index.
	\end{thm}
	\begin{rmk}
	Observe that condition \rf{aind} is meaningful outside a sufficiently large compact set $K$ because of \rf{Ghp} iii).
	\end{rmk}
		\begin{proof}
		On $\R^+$ we define the function
			\beq\label{kappa}
			\kappa(s)=1+\frac{1}{4s^2}\left[1+\frac{1}{\log^2s}\right]\,,
			\eeq
		so that inequality \rf{aind} can be rewritten as
			\beq\label{akappa}
			a(x)\leq\kappa(t(x))\frac{\left|\nabla\log G(x)\right|^2}{4}\quad\quad\hbox{on $M\setminus K$}\,.
			\eeq
		To prove the theorem we need to provide a positive solution $u$ of \rf{uind} on $M\setminus\widehat{K}$ for some compact $\widehat{K}$. Towards this aim we look for $u$ of the form
			\beq\label{ubeta}
			u(x)=\sqrt{G(x)}\beta(t(x))=\e^{-t(x)}\beta(t(x))\,,
			\eeq
		on $M\setminus\Lambda_T$ for some $T>0$ sufficiently large and with $\beta:\left[T,+\infty\right)\rightarrow\R^+$. Now a simple computation shows that $u$ satisfies
			\beq\label{lapbeta}
			\Delta u+\left[1-\frac{\ddot{\beta}}{\beta}(t(x))\right]\frac{\left|\nabla\log G(x)\right|^2}{4}u=\frac{1}{2}\frac{\Delta G}{\sqrt{G}}(x)\left[\beta(t(x))-\dot{\beta}(t(x))\right]
			\eeq
		on $M\setminus\Lambda_T$, where $\dot{\,}$ means the derivative with respect to $t$. Thus, using \rf{akappa} and \rf{lapbeta} we obtain
			\beqs
			\begin{aligned}
			\Delta u+a(x)u & \leq\left[\kappa(t(x))-1+\frac{\ddot{\beta}}{\beta}(t(x))\right]\frac{\left|\nabla\log G(x)\right|^2}{4}u\\
			& \quad\quad\quad\quad+\frac{\Delta G}{2\sqrt{G}}(x)\left[\beta(t(x))-\dot{\beta}(t(x))\right]\,.
			\end{aligned}
			\eeqs
		Hence using (\ref{Ghp}) i) we have that \rf{uind} is satisfied on $M\setminus\Lambda_T$ for $u$ as in \rf{ubeta} if we show the existence of a positive solution $\beta$ of 
			\beq\label{bode}
			\ddot{\beta}+\left[\kappa(t)-1\right]\beta=0
			\eeq
		satisfying the further requirement 
			\beq\label{be-bed}
			\beta-\dot{\beta}\geq0
			\eeq
		on $\left[T,+\infty\right)$ for some $T>0$; in other words we have to show that \rf{bode} is non-oscillatory and that (\ref{be-bed}) holds at least in a neighborhood of $+\infty$. As for non oscillation, applying Theorem 6.44 of \cite{BMR}, we see that this is the case if
			\beqs
			\kappa(t)-1\leq\frac{1}{4t^2}\left[1+\frac{1}{\log^2t}\right]
			\eeqs 
		on $\left[T,+\infty\right)$ for some $T>0$ sufficiently large. This is  guaranteed by the definition \rf{kappa} of $\kappa$. To show the validity of (\ref{be-bed}) we use the following trick. Fix $n\geq 3$ and define $\rho\in\R^+$ via the prescription
			\beq
			t=t(\rho)=\log\left(\sqrt{n-2}\rho^{\frac{n-2}{2}}\right)\,.
			\eeq
		Note that
			\beq
			t(0^+)=-\infty\,,\quad\quad t(+\infty)=+\infty\,,\quad\quad t'(\rho)=\frac{n-2}{2}\frac{1}{\rho}\quad\quad\hbox{on $\R^+$}.
			\eeq
		We then define
			\beq
			z(\rho)=\e^{-t(\rho)}\beta(t(\rho))\,.
			\eeq
		If $\beta$ is a solution of \rf{bode} on $[T,+\infty)$, having set $R=\rho(T)>0$ with $\rho(t)$ the inverse function of $t(\rho)$, $z$ satisfies
			\beq
			\left(\rho^{n-1}z'\right)'+\kappa(t(\rho))\frac{(n-2)^2}{4\rho^2}\rho^{n-1}z=0\quad\quad\hbox{on $[R,+\infty)$.}
			\eeq
		we can also fix the initial conditions
			\beq
			z(R)=1\,,\quad\quad z'(R)=0\,.
			\eeq
		Hence, since $\kappa\geq0$ on $[R,+\infty)$ a first integration of the solution $z$ of the above Cauchy problem yields
			\beqs
			z'(\rho)\leq0\quad\quad\hbox{on $[R,+\infty)$.}
			\eeqs
		But
			\beqs
			z'(\rho)=\frac{\sqrt{n-2}}{2}\frac{1}{\rho^{\frac{n}{2}}}\left\{\dot\beta(t(\rho))-\beta(t(\rho))\right\}
			\eeqs
		and therefore \rf{be-bed} is satisfied.\\
		This completes the proof of the Theorem.
		\end{proof}
	\begin{rmk}
	We have just proved that the equation
		\beq\label{HN}
		\ddot{\beta}+\frac{1}{4t^2}\left[1+\frac{1}{\log^2t}\right]\beta=0\quad\quad\hbox{on $\left[T,+\infty\right)$}
		\eeq
	(say $T\geq\e$) is non-oscillatory. This is not a consequence of the usual Hille-Nehari criterion (see \cite{Sw}). Indeed, setting $h(t)$ to denote the coefficient of the linear term in \rf{HN}, the condition of the classical criterion to guarantee the non-oscillatory character of the equation is that $h(t)\geq0$ for $t>>1$ and
		\beq\label{limHN}
		\limsup_{t\rightarrow+\infty}t\int_t^{+\infty}h(s)ds<\frac{1}{4}\,.
		\eeq
	However, in this case we have
		\beqs
		\frac{1}{4}<t\int_t^{+\infty}\frac{ds}{4s^2}<t\int_t^{+\infty}h(s)ds<\frac{1}{4}+\frac{1}{4}\int_t^{+\infty}\frac{ds}{s\log^2s}=\frac{1}{4}+\frac{1}{4\log t}
		\eeqs
	so that \rf{limHN} is not satisfied.
	\end{rmk}
We shall now see how to get non-oscillation of \rf{HN} following the idea in the proof of the mentioned Theorem 6.44 of \cite{BMR}. This will enable us to determine the asymptotic behavior of a solution $\beta$ of \rf{HN} at $+\infty$ and therefore of $u$ defined in \rf{ubeta} and solution of \rf{uind}. This will be later used in Theorem \ref{uniq2}.\\

Towards this aim we consider the function
	\beq
	w(t)=\sqrt{t}\log t
	\eeq
solution of Euler equation
	\beq
	\ddot{w}+\frac{1}{4t^2}w=0
	\eeq
on $\left[T,+\infty\right)$, $T>0$, and positive on $\left[T,+\infty\right)$ for $T>1$. Then the function
	\beq
	z=\frac{\beta}{w}
	\eeq
satisfies
	\beq\label{w2z}
	\left(w^2\dot{z}\right)^{\cdot}+\left(\kappa(t)-1-\frac{1}{4t^2}\right)w^2z=0\quad\quad\hbox{on $\left[T,+\infty\right)$}
	\eeq
for $T>>1$. Since $\frac{1}{w^2}\in\mathrm{L}^1(+\infty)$ we can define the critical curve $\chi_{w^2}$ relative to $w^2$ as in (4.21) of \cite{BMR}. A computation yields
	\beqs
	\chi_{w^2}(t)=\frac{1}{4}\frac{1}{t^2\log^2t}\quad\quad \hbox{for $t>>1$}\,,
	\eeqs
so that
	\beqs
	\kappa(t)-1-\frac{1}{4t^2}=\chi_{w^2}(t)\,.
	\eeqs
Hence from Theorem 5.1 and Proposition 5.7 of \cite{BMR} we deduce that the solution $z(t)$ of \rf{w2z} satisfies
	\beqs
	z(t)\sim\frac{C}{\sqrt{\log t}}\log\log t\quad\quad\hbox{as $t\rightarrow+\infty$}\,,
	\eeqs
for some constant $C>0$ and therefore
	\beq
	\beta(t)\sim C\sqrt{t\log t}\log\log t\quad\quad\hbox{as $t\rightarrow+\infty$}\,.
	\eeq
Using the above, we finally obtain the asymptotic behavior of $u$ in \rf{ubeta}, that is,
	\beq
	u(x)\sim \varphi(x)\quad\quad\hbox{as $x\rightarrow\infty$ in $M$}\,,
	\eeq
with
	\beq\label{phiasy}
	\varphi(x)=C\sqrt{G(x)}\sqrt{-\log\sqrt{G(x)}\log\left(-\log\sqrt{G(x)}\right)}\log\log\left(-\log\sqrt{G(x)}\right)
	\eeq
as $x\rightarrow\infty$ on $M$.\\
In particular the behavior of $u$ at infinity is known once that of $G(x)$ is known.\\
Next we prove a version of Theorem 5.20 of \cite{BMR} for equation \rf{lich}.
	\begin{thm}\label{comp2}
	Let $\M$ be a complete manifold, $a(x)$, $b(x)$, $c(x)\in\mathrm{C}^0(M)$, $\sigma>1$, $\tau<1$, and assume \rf{signbc} and \rf{bc0}. Let $\Omega$ be a relatively compact open set and assume the existence of $w\in\mathrm{C}^2(M\setminus\overline{\Omega})$ positive solution of 
		\beq\label{Lw}
		Lw=\Delta w+a(x)w\leq0\quad\quad\hbox{on $M\setminus\overline{\Omega}$}.
		\eeq
	Suppose that $u$ and $v$ are positive $\mathrm{C}^2$ solutions on $M$ of
		\beq\label{bmruv}
		\begin{cases}
		\Delta u+a(x)u-b(x)u^{\sigma}+c(x)u^{\tau}\leq 0\\
		\Delta v+a(x)v-b(x)v^{\sigma}+c(x)v^{\tau}\geq 0\,.
		\end{cases}
		\eeq
	If 
		\beq\label{uvw}
		u-v=o(w)\quad\quad\hbox{as $x\rightarrow\infty$}\,,
		\eeq
	then $v\leq u$ on $M$.
	\end{thm}
		\begin{proof}
		The idea of the proof is the same as that of Theorem 5.20 of \cite{BMR}. We report it here for the sake of completeness and for some minor differences. First we extend $w$ to a positive function $\widetilde{w}$ on $M$. Towards this end let $\Omega^{\prime}$ be a relatively compact open set such that $\overline{\Omega}\subset\Omega^{\prime}$ Fix a cut-off function $\psi$, $0\leq\psi\leq1$ such that $\psi\equiv1$ on $\Omega$ and $\psi\equiv0$ on $M\setminus\Omega^{\prime}$. Define $\widetilde{w}=\psi+(1-\psi)w$. Note that $\widetilde{w}>0$ on $M$ and $\widetilde{w}=w$ on $M\setminus\overline{\Omega^{\prime}}$ so that $L\widetilde{w}\leq0$ on $M\setminus\overline{\Omega^{\prime}}$. For notational convenience we write again $w$ and $\Omega$ in place of $\widetilde{w}$ and $\Omega^{\prime}$, but this time $w>0$ on $M$.\\
		Let $\ep>0$ and define $u_{\ep}=u+\ep w$ on $M$. Then $u_{\ep}$ is a solution on $M$ of
			\beqs
			\Delta u_{\ep}+a(x)u_{\ep}\leq b(x)u^{\sigma}-c(x)u^{\tau}+\ep Lw\,.
			\eeqs
		Therefore, interpreting the differential inequality in the weak sense, we have that for each $\varphi\in\mathrm{Lip}_{loc}(M)$, $\varphi\geq0$
			\beqs
			-\int_M\g{\nabla u_{\ep}}{\nabla\varphi}+\int_Ma(x)u_{\ep}\varphi\leq\int_Mb(x)u^{\sigma}\varphi-\int_Mc(x)u^{\tau}\varphi+\ep\int_M\varphi Lw\,.			
			\eeqs
		Now, by the second Green formula
			\beqs
			\int_M\varphi Lw=\int_Ma(x)w\varphi+\int_Mw\Delta\varphi=\int_MwL\varphi
			\eeqs
		and therefore we can rewrite the above inequality as
			\beq\label{uphi}
			-\int_M\g{\nabla u_{\ep}}{\nabla\varphi}+\int_Ma(x)u_{\ep}\varphi\leq\int_Mb(x)u^{\sigma}\varphi-\int_Mc(x)u^{\tau}\varphi+\ep\int_MwL\varphi \,.			
			\eeq
		Similarly, interpreting the second differential inequalitty of \rf{bmruv} in the weak sense
			\beq\label{vphi}
			-\int_M\g{\nabla v}{\nabla\varphi}+\int_Ma(x)v\varphi\geq\int_Mb(x)v^{\sigma}\varphi-\int_Mc(x)v^{\tau}\varphi \,,			
			\eeq
		with $\varphi$ as above.\\
		Next, by contradiction suppose that
			\beqs
			\Gamma=\left\{x\in M : v(x)>u(x)\right\}\neq\emptyset\,.
			\eeqs
		Then, for $\ep>0$ sufficiently small
			\beq\label{Gamep}
			\Gamma_{\ep}=\left\{x\in M : v(x)>u_{\ep}(x)\right\}\neq\emptyset\,.
			\eeq
		We now consider the Lipschitz function $\gamma_{\ep}=\left(v^2-u_{\ep}^2\right)_+$. Condition \rf{uvw} implies that $\gamma_{\ep}$ has compact support in $M$ and it is not identically zero because of \rf{Gamep}. Thus the functions $\varphi_1=\frac{\gamma_{\ep}}{u_{\ep}}$ and $\varphi_2=\frac{\gamma_{\ep}}{v}$ are admissible, respectively for \rf{uphi} and \rf{vphi}. Substituting we have
			\beqs
			-\int_M\g{\frac{\nabla u_{\ep}}{u_{\ep}}}{\nabla\gamma_{\ep}}-\frac{\left|\nabla u_{\ep}\right|^2}{u_{\ep}^2}\gamma_{\ep}-a(x)\gamma_{\ep}\leq\int_Mb(x)\frac{u^{\sigma}}{u_{\ep}}\gamma_{\ep}-c(x)\frac{u^{\tau}}{u_{\ep}}\gamma_{\ep}+\ep wL\left(\frac{\gamma_{\ep}}{u_{\ep}}\right)\,,		
			\eeqs
		and
			\beqs
			-\int_M\g{\frac{\nabla v}{v}}{\nabla\gamma_{\ep}}-\frac{\left|\nabla v\right|^2}{v^2}\gamma_{\ep}-a(x)\gamma_{\ep}\geq\int_Mb(x)v^{\sigma-1}\gamma_{\ep}-c(x)v^{\tau-1}\gamma_{\ep}\,.			
			\eeqs
		Thus, subtracting the second from the first we deduce
			\beqs
			\begin{aligned}
			& -\int_{\Gamma_{\ep}}\g{\frac{\nabla u_{\ep}}{u_{\ep}}-\frac{\nabla v}{v}}{\nabla\gamma_{\ep}} + \int_{\Gamma_{\ep}}	\left(\frac{\left|\nabla u_{\ep}\right|^2}{u_{\ep}^2}-\frac{\left|\nabla v\right|^2}{v^2}\right)\gamma_{\ep} \leq \\
			&\quad\quad\quad\leq\int_{\Gamma_{\ep}}b(x)\left(\frac{u^{\sigma}}{u_{\ep}}-v^{\sigma-1}\right)\gamma_{\ep}- \int_{\Gamma_{\ep}}c(x)\left(\frac{u^{\tau}}{u_{\ep}}-v^{\tau-1}\right)\gamma_{\ep}+\ep\int_MwL\left(\frac{\gamma_{\ep}}{u_{\ep}}\right)\,. 
			\end{aligned}
			\eeqs
		Inserting the expression for $\gamma_{\ep}$ and rearranging, we finally have
			\beq\label{rearr}
			\begin{aligned}		
			& \int_{\Gamma_{\ep}}\left|\nabla u_{\ep}-\frac{u_{\ep}}{v}\nabla v\right|^2-\left|\nabla v-\frac{v}{u_{\ep}}\nabla u_{\ep}\right|^2 \leq \int_{\Gamma_{\ep}}b(x)\left(\frac{u^{\sigma}}{u_{\ep}}-v^{\sigma-1}\right)\gamma_{\ep}-\\
			& \quad\quad\quad\quad\quad\quad  - \int_{\Gamma_{\ep}}c(x)\left(\frac{u^{\tau}}{u_{\ep}}-v^{\tau-1}\right)\gamma_{\ep}+\ep\int_MwL\left(\frac{\gamma_{\ep}}{u_{\ep}}\right)\,.
			\end{aligned}
			\eeq
		Let $V$ be a relatively compact open set with smooth boundary such that $\overline{\Omega}\subset V$ and let $\psi$, $0\leq\psi\leq1$ be a cut-off function such that $\psi\equiv1$ on $\Omega$ and $\psi\equiv 0$ on $M\setminus \overline{V}$. Then, using again the second Green formula and \rf{Lw} we have
			\beqs
			\begin{aligned}
			\int_Mw L\left(\frac{\gamma_{\ep}}{u_{\ep}}\right)& = \int_Mw L\left(\psi\frac{\gamma_{\ep}}{u_{\ep}}\right)+\int_Mw L\left(\left(1-\psi\right)\frac{\gamma_{\ep}}{u_{\ep}}\right)\\
			& = \int_M w L\left(\psi\frac{\gamma_{\ep}}{u_{\ep}}\right) + \int_M\left(1-\psi\right)\frac{\gamma_{\ep}}{u_{\ep}} L w\\
			& \leq \int_M w L \left(\psi\frac{\gamma_{\ep}}{u_{\ep}}\right)\,.
			\end{aligned}
			\eeqs
		Now since $u_{\ep}$ is bounded from below by a positive constant on $\overline{V}$, by applying the dominated convergence theorem we deduce that
			\beqs
			\lim_{\ep\rightarrow 0}\ep\left|\int_M w L\left(\psi\frac{\gamma_{\ep}}{u_{\ep}}\right)\right| \leq \lim_{\ep\rightarrow 0}\ep\left[\int_V\left|\nabla w\right|\left|\nabla\left(\psi\frac{\gamma_{\ep}}{u_{\ep}}\right)\right|+\left|a(x)w\psi\frac{\gamma_{\ep}}{u_{\ep}}\right|\right]=0\,.
			\eeqs
		Therefore, letting $\ep\rightarrow 0$ in \rf{rearr}, using Fatou's lemma and the last two inequalities, we get
			\beq\label{0uv0}
			\begin{aligned}
			0 & \leq \int_{\Gamma}\left|\nabla u-\frac{u}{v}\nabla v\right|^2 + \int_{\Gamma}\left|\nabla v-\frac{v}{u}\nabla u\right|^2\\ 
			& \leq \int_{\Gamma}b(x)\left(u^{\sigma-1}-v^{\sigma-1}\right)\left(v^2-u^2\right) - \int
_{\Gamma}c(x)\left(u^{\tau-1}-v^{\tau-1}\right)\left(v^2-u^2\right)\\
			& \leq 0\,.
			\end{aligned}
			\eeq
		Therefore $\frac{v}{u}$ is constant on any connected component of $\Gamma$. Clearly $\Gamma$ must have no boundary because otherwise letting $x\rightarrow\partial\Gamma$ we would deduce $u=v$ on $\Gamma$ which is a contradiction. By connectedness $v=Au$ on $M$ for some $A>1$ and inserting into \rf{0uv0} we deduce
			\beqs
			\int_{\Gamma}b(x)\left(1-A^{\sigma-1}\right)\left(1-A^2\right)u^{\sigma+1} + \int
_{\Gamma}c(x)\left(A^{\tau-1}-1\right)\left(1-A^2\right)u^{\tau+1} \equiv 0\,.
			\eeqs
		Since $u>0$, this contradicts assumptions \rf{signbc} and \rf{bc0}. Hence $\Gamma=\emptyset$ that is, $v\leq u$ on $M$.
		\end{proof}
Thus, considering $\varphi$ defined in \rf{phiasy}, as a consequence of Theorem \ref{comp2}, Theorem \ref{thmind}, and the subsequent discussion we have
	\begin{thm}\label{uniq2}
	Let $\MG$ as in \rf{Ghp}, $a(x)$, $b(x)$, $c(x)\in\mathrm{C}^0(M)$, $\sigma>1$, $\tau<1$, and assume \rf{signbc}, \rf{bc0}, and
		\beqs
		a(x)\leq\left\{1+\frac{1}{\log^2 G(x)}\left[1+\frac{1}{\log^2\left(-\log\sqrt{G(x)}\right)}\right]\right\}\frac{\left|\nabla\log G(x)\right|^2}{4}	
		\eeqs
	outside a compact set. If $u$ and $v$ are positive $\mathrm{C}^2$ solutions of
		\beqs
		\Delta u+a(x)u-b(x)u^{\sigma}+c(x)u^{\tau}=0
		\eeqs
	such that
		\beqs
		u(x)-v(x)=o(\varphi(x))\quad\quad\hbox{as $x\rightarrow\infty$}
		\eeqs
	with $\varphi(x)$ as in \rf{phiasy}, then $u\equiv v$ on $M$. 
	\end{thm}
It is reasonable that if we strenghten the upper bound \rf{akappa} on $a(x)$ the growth of $u$ defined in \rf{ubeta} should improve in \rf{phiasy}.\\
For the sake of simplicity let us suppose
	\beq
	a(x)\leq\lambda\frac{\left|\nabla\log G(x)\right|^2}{4}
	\eeq
on $M\setminus K$ for some constant $\lambda\in\left(-\infty,1\right]$. We proceed as in the proof of Theorem \ref{thmind} to arrive to \rf{bode} that now reads
	\beq
	\ddot{\beta}+\left[\lambda-1\right]\beta=0
	\eeq
on $\left[T,+\infty\right)$ for some $T>0$. Positive solutions of the above are immediately obtained. Indeed, for $\lambda=1$ we let $\beta(t)=Ct$ for some constant $C>0$ while for $\lambda\in\left(-\infty,1\right)$ we let $\beta(t)=C\e^{\sqrt{1-\lambda}t}$, $C>0$. Thus the positive solution $u(x)$ of $Lu\leq 0$ given in \rf{ubeta} satisfies
	\beq
	u(x)\sim
		\begin{cases}
		C\sqrt{G(x)}\log\frac{1}{G(x)}&\hbox{if $\lambda=1$}\\
		C G(x)^{\frac{1-\sqrt{1-\lambda}}{2}}&\hbox{if $\lambda\in\left(-\infty,1\right)$}
		\end{cases}
	\eeq
as $x\rightarrow\infty$ for some constant $C>0$.\\
Thus, going back to Theorem \ref{uniq2} we obtain the following version
	\begin{thm}
	Let $\MG$ as in \rf{Ghp}, $a(x)$, $b(x)$, $c(x)\in\mathrm{C}^0(M)$, $\sigma>1$, $\tau<1$, and assume \rf{signbc}, \rf{bc0}, and
		\beqs
		a(x)\leq\lambda\frac{\left|\nabla\log G(x)\right|^2}{4}	
		\eeqs
	outside a compact set, for some constant $\lambda\in\left(-\infty,1\right]$. If $u$ and $v$ are positive $\mathrm{C}^2$ solutions of
		\beqs
		\Delta u+a(x)u-b(x)u^{\sigma}+c(x)u^{\tau}=0
		\eeqs
	such that
		\beqs
		u(x)-v(x)=
		\begin{cases}
		o\left(\sqrt{G(x)}\log\frac{1}{G(x)}\right)&\hbox{if $\lambda=1$}\\
		o\left(G(x)^{\frac{1-\sqrt{1-\lambda}}{2}}\right)&\hbox{if $\lambda\in\left(-\infty,1\right)$}
		\end{cases}		
		\quad\quad\hbox{as $x\rightarrow\infty$}\,,
		\eeqs
	then $u\equiv v$ on $M$. 
	\end{thm}
As a final remark we observe that finiteness of the index of $L=\Delta+a(x)$ can be also deduced by the validity of a Sobolev-type inequality on $M$. Indeed, according to Lemma 7.33 of \cite{PRSvan}, the validity of (\ref{sobolev}) and the assumption
	\beqs
	a_+(x)\in\mathrm{L}^{{1}\slash{\alpha}}(M)
	\eeqs 
imply that $L$ has finite index.

\end{document}